\documentclass[10pt,twoside,english]{article}

\usepackage[toc,page]{appendix} 
\usepackage{latexsym}
\usepackage{amsmath,amsthm}
\usepackage{amsfonts}
\usepackage{amssymb}
\usepackage{graphicx}
\usepackage{color}
\usepackage{epsfig}
\usepackage{hyperref}
\hypersetup{colorlinks=true, citecolor=blue}

\def\leq{\leqslant}
\def\geq{\geqslant}
\def\N{\mathbb{N}}
\def\R{\mathbb{R}}

\newtheorem{Proposition}{Proposition}[section]
\newtheorem{example}{Example}[section]
\newtheorem{lemma}[Proposition]{Lemma}
\newtheorem{definition}[Proposition]{Definition}

\newtheorem{theorem}[Proposition]{Theorem}

\begin{document}
\title{On the Asymptotic Properties of Piecewise Contracting Maps}
\author{E.\ Catsigeras$^1$, P.\ Guiraud$^2$, A.\ Meyroneinc$^3$ and E.\ Ugalde$^4$}
\maketitle
\begin{center}
\begin{small}
$^1$ Instituto de Matem\'{a}tica y Estad\'istica Rafael Laguardia, Universidad de la Rep\'{u}blica,\\
Montevideo, Uruguay\\
\texttt{eleonora@fing.edu.uy}

$^2$ Centro de Investigaci\'on y Modelamiento de Fen\'omenos Aleatorios Valpara\'iso, Universidad de Valpara\'{\i}so,\\
Valpara\'{\i}so, Chile\\
\texttt{pierre.guiraud@uv.cl}

$^3$ Departamento de Matem\'aticas, Instituto Venezolano de Investigaciones Cient\'ificas,\\
Apartado 20632, Caracas 1020A, Venezuela\\
\texttt{ameyrone@ivic.gob.ve}

$^4$ Instituto de F\'isica, Universidad Aut\'onoma de San Luis Potos\'i,\\
78000 San Luis Potos\'i, Mexico\\
\texttt{ugalde@ifisica.uaslp.mx}
\end{small}
\end{center}

\begin{abstract}

We study the asymptotic dynamics of maps which are piecewise contracting on a compact space. These maps are Lipschitz continuous, with Lipschitz constant smaller than one, when restricted to any piece of a finite and dense union of disjoint open pieces. 
We focus on the topological and the dynamical properties of the (global) attractor of the orbits that remain in this union. As a starting point, we show that the attractor consists of a finite set of periodic points when it does not intersect the boundary of a contraction piece, which complements similar results proved for more specific classes of piecewise contracting maps. 
Then, we explore the case where the attractor intersects these boundaries by providing examples that show  the rich phenomenology of these systems. Due to the discontinuities, the asymptotic behaviour is not always properly represented by the dynamics in the attractor.
Hence, we introduce generalized orbits to describe the asymptotic dynamics and its recurrence and transitivity properties.
Our examples include transitive and recurrent attractors, that are either finite, countable, or a disjoint union of a Cantor set and a countable set.
We also show that the attractor of a piecewise contracting map is usually a Lebesgue measure-zero set, and we give conditions ensuring that it is totally disconnected. Finally, we provide an example of piecewise contracting map with positive topological entropy and whose attractor is an interval.

\end{abstract}



\section{Introduction}

Piecewise continuous dynamical systems have been considered as an alternative to classical continuous models for the study of  nonlinear phenomena from nature and engineering. Within this class, piecewise contracting systems have been used to model
the dynamics of dissipative systems interacting in a nonlinear way. For instance, they  appear in biological networks as discrete time models \cite{C08,CFLM06,LU06}, or as Poincar\'e return maps of continuous time models defined by piecewise continuous vector fields \cite{C10,CG10,EDW07,MS90}. More generally, the dynamics of many dissipative hybrid systems can be described by a piecewise contracting map (see for instance \cite{GFB12,R90}, the introduction of \cite{NP12}, and references therein). In some cases, piecewise continuous dynamical systems are more convenient for an exhaustive mathematical analysis. Nevertheless, the presence of discontinuities can produce phenomena that do not always exist in a continuous framework and which deserve detailed studies. Because of their good properties far from the discontinuities, piecewise contracting maps appear as suitable study models.

Several results have been obtained for particular families of piecewise contracting maps. Let us mention some of them which we consider representative. 
In \cite{GT88}, Gambaudo and Tresser gave a complete description of the possible asymptotic dynamics of quasi-contractions. These systems, which act on an arbitrary metric space, generalize one-dimensional piecewise contractions with two contracting domains by maintaining only their essential features. 
The characterization of the symbolic dynamics of these maps obtained in \cite{GT88} shows that their limit set is either composed of at most two periodic orbits or is a Cantor set supporting a quasi-periodic dynamics.
Some results in the same direction have also been obtained for piecewise contracting maps with more than two contraction pieces. 
In \cite{B06}, Br\'emont studied contracting interval exchange transformations
and proved, among other results, that the asymptotic dynamics is typically concentrated
in a finite number of periodic attractors whose number can be bounded by a linear function
of the number of monotonicity intervals of the transformation. This result was further refined
by Nogueira and Pires in \cite{NP12}, where they established that 
the number of periodic attractors is at most equal to the number of monotonicity 
intervals. 
In larger dimensions, piecewise contracting affine maps have received a particular attention. We can mention the result by Bruin and Dean \cite{BD09}, according to which the asymptotic dynamics of piecewise-affine contractions of the complex plane is typically supported by a finite number of periodic orbits. 
This result was also proved in \cite{C08,LU06} for a particular class of affine piecewise contractions on $\R^n$ introduced in \cite{CFLM06} as models of genetic regulatory networks. In addition it has been shown that the symbolic complexity of these systems (including non typical ones) grows sub-exponentially with time \cite{CMU08, LU06}. 
Finally, general piecewise contracting maps on $\R^n$ were studied in \cite{C10, CG10,R90}  
and the generic character of a periodic asymptotic dynamics was proved in \cite{CB11} under some 
separation property asumption (which is verified by injective maps, for instance).

Most works on piecewise contracting maps have studied specifics (affine) maps, which allowed to obtain detailed results concerning their periodic or quasi-periodic asymptotic behaviours. 
In the present paper, we propose to study the asymptotic dynamics of piecewise contracting maps in a larger framework. Our purpose is on one hand to obtain properties that are shared by all piecewise contracting maps defined on a compact set, and in the other hand to illustrate the rich phenomenology emerging from the presence of discontinuities by providing numerous examples, most of them contrasting with the (quasi-)periodic case.

Our discussion is organized as follows. In Section \ref{SECDEF}, we give the principal definitions and we review some of the classical notions used in the study of asymptotic dynamics, as those of attractor, limit set, and non-wandering set, in order to adapt them to discontinuous maps. 
In Section \ref{SECASYMP}, we show that when the attractor does not intersect the set of discontinuities, it is composed of a finite number of
periodic orbits. This result is true for any piecewise contracting map defined on a compact set and thus is independent of the particularities of the map (dimension of the space, number and geometry of the contractions pieces, in particular). In the sequel of the paper, we focus on the case where the attractor contains discontinuity points (which justifies our generalization of the asymptotic sets introduced in Section \ref{SECDEF}). In this situation the strong recurrence properties of the periodic case are lost in general, since attractor, limit set and non-wandering can be different. In Section \ref{SECDYNATT}, we show that the discontinuities can give rise to an asymptotic behaviour  which does not coincide with the dynamics of the map on the attractor. 
This is why we introduce the notion of ghost orbit, which are generalized orbits allowing to describe the asymptotic dynamics as well as to define recurrence and transitivity in the attractor. 
Section \ref{SECCOMP} deals with the dynamical complexity of piecewise contracting maps. There, we show that the total disconnectedness of the attractor is related to a low complexity of the dynamics and we give examples of piecewise contracting maps with a connected attractor. Our final example is a piecewise contracting map with positive topological entropy whose attractor is a Cantor set or an interval. In the concluding section we present our final comments and we point some directions for future research.


\section{Definitions}\label{SECDEF}

\begin{definition}\label{definitionPWC}
{\bf Piecewise contracting map.} \em Let $(X,d)$ be a compact metric space. We say that a map $f: X \to X$ is {\em piecewise contracting} if there exists a finite collection of $N \geq 2$ non-empty open subsets $\{X_i\}_{i=1}^N$, such that:

\noindent 1) $X = \bigcup_{i=1}^{N} \overline{X_i}$ and $X_i\cap X_j=\emptyset$ for all $i\neq j\in \{1,\ldots,N\}$.

\noindent 2) There exists $\lambda \in (0,1)$, called a \em contraction rate\em, such that for any $i \in \{1,\ldots,N\}$ we have
\begin{equation}\label{CONTRACTION}
d(f(x),f(y)) \leq \lambda \, d(x,y)
\end{equation}
if $x,y\in X_i$.

\noindent 3) The set $\widetilde X := \bigcap _{n= 0}^{+ \infty} f^{-n}(X\setminus\Delta) $, where $\Delta := X \setminus (\bigcup_{i=1}^N X_i)$, is non-empty.\em
\end{definition}

\noindent
If the set $\Delta$ is non-empty, then it contains all the discontinuity points of the piecewise contracting map. Let us note however that each restriction $f|_{X_i}$ of $f$ to a piece $X_i$ admits a continuous extension $f_i:\overline{X_i}\to X$ which we call the continuous extensions of $f$. In this paper we study the asymptotic dynamics of the points whose orbit never intersects $\Delta$, namely the points of ${\widetilde X}$. Nevertheless, we describe this  dynamics by mean of compact asymptotic sets that can intersect the set $\Delta$. The set ${\widetilde X}$ is dense in $X$ under mild hypothesis on $f$ (see Proposition \ref{TheoremDensidadXtildePierre}), and it is usually not compact.

\begin{definition} \label{ATTRACTOR}
{\bf Attractor $\Lambda$.} \em Let $f:X\to X$ be a piecewise contracting map and consider the family of transformations $\{F_i\}_{i\in\{1,\dots,N\}}$ defined for all $A\subset X$ by $F_i(A)=\overline{f(A\cap X_i)}$.
We say that $A$ is an \em atom of generation $n\geq 1$ \em if
\begin{equation} \label{EquationAtom}
A = F_{i_n}\circ F_{i_{n-1}}\circ\dots\circ F_{i_1}(X)
\end{equation}
for some $i_1,i_2,\dots,i_n\in\{1,\dots,N\}$.
Denote by $\mathcal{A}_n$ the set of all the atoms of generation $n$, and for all $n\geq 1$ let
$\Lambda_n:=\bigcup\limits_{A\in\mathcal{A}_n} A$. We call
$\Lambda:=\bigcap\limits_{n\geq 1}\Lambda_n$ the {\em attractor} of $f$.\em
\end{definition}

\noindent
Note that any atom is a subset of an atom of the previous generation. It implies that $\Lambda_{n+1} \subset \Lambda_n$ for all $n \geq 1$. Since $\widetilde X$ is non-empty, there exist a non-empty atom of any generation.
The sets $\Lambda_n$, and therefore $\Lambda$, are  non-empty and compact.
Here, we use the terminology \em attractor \em in a formal analogy with the attractor of iterated functions systems \cite{FAL}. Indeed, $\Lambda$
can also be written as $\Lambda=\bigcap_{n\in\N}G^n(X)$, where for any compact set $A\subset X$, the map $G$
is defined by $G(A)=\bigcup_{i=1}^Nf_i(A\cap \overline{X_i})$. The invariance of the attractor by $G$ implies that $f_i(x)\in\Lambda$ for any $x\in\Lambda$ and $i\in\{1,\dots,N\}$ such
that $x\in \overline{X_i}$. Unless it does not intersect $\Delta$, the attractor is not necessarily invariant by $f$, but for any $x\in\Lambda$ there is always a continuous extension $f_i$ such that $f_i(x)\in\Lambda$.

We also consider other asymptotic sets traditionally used for continuous maps, but we adapt their definitions to deal with discontinuities.

\begin{definition}\label{definitionNoErrante}
{\bf Non-wandering set $\Omega$.} \em A point $x \in X$ is {\em non-wandering} if for all $\epsilon >0$ there exists a sequence $\{n_k\}_{k \in \mathbb{N}}$ going to infinity and such that the ball $B(x,\epsilon)$ of center $x$ and radius $\epsilon$ satisfies:
\begin{equation}\label{equationNoErrante}
 f^{n_k}(B(x, \epsilon) \cap \widetilde X) \bigcap B(x, \epsilon) \neq \emptyset \qquad \forall \, k \in \mathbb{N}.
\end{equation}
We denote by $\Omega$, and call the {\em non-wandering set}, the set of all the non-wandering points.
A point $x\in X$ is {\em wandering } if $x \not \in \Omega$.\em
\end{definition}

\noindent Contrarily to continuous maps, the existence for all $\epsilon>0$ of a $n>0$ such that  $f^{n}(B(x,\epsilon)\cap \widetilde X)$ and  $B(x, \epsilon)$ have a non-empty intersection is not equivalent with the existence for all $\epsilon>0$ of a sequence $\{n_k\}_{k\in\N}$ such that $x$ satisfies (\ref{equationNoErrante}). This equivalence is true if $x\in\widetilde X$, but not necessarily otherwise (see the final comment of Example \ref{ejemploNuevoFebrero2011}). This why we use the sequence  $\{n_k\}_{k\in\N}$ in order to ensure a sufficiently strong recurrence property to the non-wandering points belonging to $\Delta$. However, as for continuous maps, the set of the wandering points is open an thus $\Omega$ is closed.

\begin{definition}\label{definitionlimitSet}
{\bf Limit set $L$.} \em We say that $y\in X$ is an {\em $\omega$-limit point} of $x\in\widetilde X$ if there exists a sequence $\{n_k\}_{k \in \mathbb{N}}$ going to infinity and such that $\lim\limits_{k\to\infty} f^{n_k}(x)=y$.
We denote by $\omega(x)$, and call {\em $\omega$-limit set of $x$},  the set of all the $\omega$-limit points of $x$. We call  
 ${L:=\overline{\bigcup_{x\in\widetilde{X}}\ \omega(x)}}$ {\em the limit set  of $f$}.\em
\end{definition}

\noindent Since $\widetilde X \neq \emptyset$ and $X$ is compact, $\omega(x) \neq \emptyset$ for all  $x \in  \widetilde X$  and $L \neq \emptyset$.

Note that two piecewise contracting maps that coincide on $X\setminus\Delta$ share the same asymptotic sets, since the attractor, the non-wandering set and the limit set are entirely determined  by the orbits of the points of $\widetilde X$. In order to study the recurrence and transitivity properties on the attractor, when it intersects $\Delta$, we will need to introduce generalized orbits (see Section \ref{SECDYNATT}).

\section{On the asymptotic sets}\label{SECASYMP}

As a starting point, we state a general property of piecewise contracting maps concerning the periodicity of the dynamics in the attractor.

\begin{theorem}\label{SEPARATION}
Let $\Delta$ be the set of discontinuities of a piecewise contracting map and $\Lambda$ be its attractor.
If $\Delta = \emptyset$ or if $d(\Lambda,\Delta)\neq 0$, then $\Lambda$ is a finite union of periodic orbits.
\end{theorem}

\noindent We prove Theorem \ref{SEPARATION} in Section \ref{proofTheoremSEPARATION}.
The finite periodicity of the limit set has been reported for special classes of piecewise contracting maps: one-dimensional injective maps \cite{B06}, two-dimensional affine maps \cite{BD09}, arbitrary dimensional affine maps with convex continuity pieces \cite{C08,CMU08,LU06} and more general real maps satisfying a so-called separation property \cite{C10,CB11} or
with two contraction pieces \cite{GT88}. 
In each of these works, the specificity of the maps allows to obtain additional results, such as the genericity of the asymptotic periodicity or a bound on the number of periodic orbits. Nevertheless, one does not need more than the hypothesis of Theorem \ref{SEPARATION} to ensure the finite periodicity of the attractor of piecewise contracting maps as general as those of Definition \ref{definitionPWC}. 

The principal idea of the proof is the same as in the particular cases. 
However, for the maps we consider, the successive iterations  (namely, the atoms in Definition \ref{ATTRACTOR}), may have infinitely many connected components. This situation cannot occur for specific maps as considered before, for example, if all the  pieces are convex and embedded in $\mathbb{R}^n$ and the map is an affinity in each piece. 
When the atoms have an infinite number of connected components, there is no guarantee a priori, that each connected component of the atom has an iterate which returns inside itself. This avoids an immediate reduction of the proof of Theorem \ref{SEPARATION} using the Banach fixed point theorem, for instance.

When the attractor of a piecewise contracting map does not contain discontinuity points, the effect of the discontinuities is the emergence of a finite number of periodic orbits in the attractor.
Therefore, the discontinuities increase the complexity of the dynamics 
compared with pure contractions, but in both cases the dynamics is regular.
From now on we are interested in the situation where there are always orbits in ${\widetilde X}$ arbitrarily close to discontinuity points. We will see that most of the characteristics of the periodic case are lost and a rich phenomenology appears.

\begin{Proposition}\label{OMEGALAMBDA}
Let $f:X\to X$ be a piecewise contracting map.
Then the chain of inclusions $L \subset \Omega \subset \Lambda$ holds.\footnote{It is actually true for any piecewise continuous map, disregarding whether it is piecewise contracting.}
Besides, if $X$ is locally connected, then any non-wandering point in the interior of $\widetilde X$ is periodic.
\end{Proposition}

We prove Proposition \ref{OMEGALAMBDA} in Section \ref{proofTheoremOMEGALAMBDA}. An obvious consequence of Theorem \ref{SEPARATION} is that the limit set, the non-wandering set and the attractor are finite and coincide when the attractor does not intersect the set of discontinuities of the map. As shown by the following examples, when the attractor contains discontinuity points, this equality does not hold in general. In Example \ref{OpenExample2right}, the non-wandering set  and the limit set are equal, but the attractor contains an infinite number of wandering points in the interior of $\widetilde{X}$.  

\begin{example} \label{OpenExample2right}\em Consider the square $X = [0,1]^2 $ and the two pieces $X_1=\{(x,y) \in X: y<x^2\}$ and $X_2=\{(x,y) \in X : y>x^2\}$ with the parabolic boundary $\Delta=\{(x,y) \in X : y = x^2\}$. Let $\lambda\in(0,1)$ and define on $\overline{X_1}$ the map $f_1(x,y)=\lambda (x,y)$ and on $\overline{X_2}$ the map $f_2(x,y) = \lambda (x,y-1) + (0,1)$. Let $f : X \to X$ be such that $f|_{X_i}=f_i|_{X_i}$ for all $i\in\{1, 2\}$ (see the left frame of Figure \ref{FigureEjemploNuevoFebrero2011}.) One can check that $L = \Omega = \{(0,0),(0,1)\} $ and $ \Lambda = \{(0,1)\} \bigcup \{(0,1-\lambda^n) : n\in\N\}$. Thus, the attractor is countably infinite, but the non-wandering and the limit sets are finite. 
\end{example}

\begin{figure}[h]
\setlength{\unitlength}{0.5truecm}
\begin{center}
\begin{picture}(28,9)(0,0)

\put(0,1){\vector(0,1){8.5}}
\put(0,1){\vector(1,0){8.5}}
\put(3.7,7){$X_2$}
\put(6.8,5){$X_1$}
\put(8,9){\line(-1,0){8}}
\put(8,9){\line(0,-1){8}}
{\color[rgb]{0,0,1} \qbezier(0,1)(4,1)(8,9)}  

\put(0,1){\circle*{.15}}
\put(-0.3,0.3){$0$}
\put(0,9){\circle*{.15}}
\put(-0.7,8.8){$1$}
\put(3,1){\circle*{.15}}
\put(2.7,0.3){$\lambda$}

\put(7.9,0.3){$1$}
\put(0,6){\circle*{.15}}
\put(0.1,5.5){{\tiny $1-\lambda$}}
\put(0,7.7){\circle*{.15}}
\put(0.1,7.5){\tiny $1-\lambda^2$}
\put(0,8.5){\circle*{.15}}
\put(0.1,8.3){\tiny $1-\lambda^3$}

\multiput(0,8.82)(0,0.12){2}{\circle*{0.15}}

\put(0,1){\line(1,0){3}}
\put(3,1){\color[rgb]{1,0,0}\line(0,1){3}}
{\color[rgb]{1,0,0}\qbezier(0,1)(1.5,1)(3,4)}
\put(0.3,3){\small $f_1(X_1)$}
\put(1.4,2.8){\vector(1,-1){.8}}

{\color[rgb]{1,0,0}\qbezier(0,6)(1.7,6)(3,9)}
\put(2.3,6){\small $f_2(X_2)$}
\put(2.9,6.6){\vector(-1,1){0.8}}


\put(12,1){\line(0,1){8}}
\put(28,9){\line(0,-1){8}}
\put(28,9){\line(-1,0){16}}
\put(14.7,6){$X_3$}
{\color[rgb]{0,0,1} \qbezier(20,9)(20.4,5.6)(12,1)}  

\put(22.5,5.5){$X_2$}
{\color[rgb]{0,0,1} \qbezier(20,1)(24,1)(28,9)}  
\put(26.7,5.2){$X_1$}
\put(12,1){\vector(1,0){16.5}}
\put(11.6,0.4){$-1$}
\put(20,1){\vector(0,1){8.5}}
\put(20,1){\circle*{0.15}}
\put(19.9,0.4){$0$}

\put(27.8,0.4){$1$}

\put(22.5,1){\circle*{.15}}
\put(22.3,0.4){$\lambda$}

\put(20,6.5){\circle*{.15}}
\put(20.2,6.35){\tiny $1-\lambda$}

\put(20,8.2){\circle*{.15}}
\put(20.2,8.0){\tiny $1-\lambda^2$}

\put(20,8.75){\circle*{.15}}
\put(20.2,8.6){\tiny $1-\lambda^3$}

\multiput(20,8.94)(0,0.06){2}{\circle*{0.15}}

\put(22.5,1){\color[rgb]{1,0,0}\line(0,1){2.5}}
{\color[rgb]{1,0,0}\qbezier(20,1)(21.1,1)(22.5,3.5)}
\put(20.1,2.8){\small $f_1(X_1)$}
\put(21.1,2.5){\vector(1,-1){0.8}}

\put(28,3.5){\color[rgb]{1,0,0}\line(-1,0){2.5}}
{\color[rgb]{1,0,0}\qbezier(25.5,3.5)(25.5,2.5)(28,1)}
\put(25.9,2.8){\small $f_3(X_3)$}

{\color[rgb]{1,0,0}\qbezier(20,9)(20,7.8)(17.5,6.5)}
\put(20,6.5){\color[rgb]{1,0,0}\line(-1,0){2.5}}
{\color[rgb]{1,0,0}\qbezier(20,6.5)(21.1,6.5)(22.5,9)}
\put(22.1,7.5){\small $f_2(X_2)$}
\put(22,7.6){\vector(-1,0){0.8}}
\end{picture}

\caption{The contraction pieces ($\Delta$ is shown in blue) and their images (boundaries in red) with  points of $\Lambda$ for two piecewise contracting maps.  
Left frame: $L = \Omega \subsetneq \Lambda$ (Example \ref{OpenExample2right}).
Right frame: $L \subsetneq \Omega = \Lambda$ (Example \ref{ejemploNuevoFebrero2011}).
}
\label{FigureEjemploNuevoFebrero2011}
\end{center}
\end{figure}
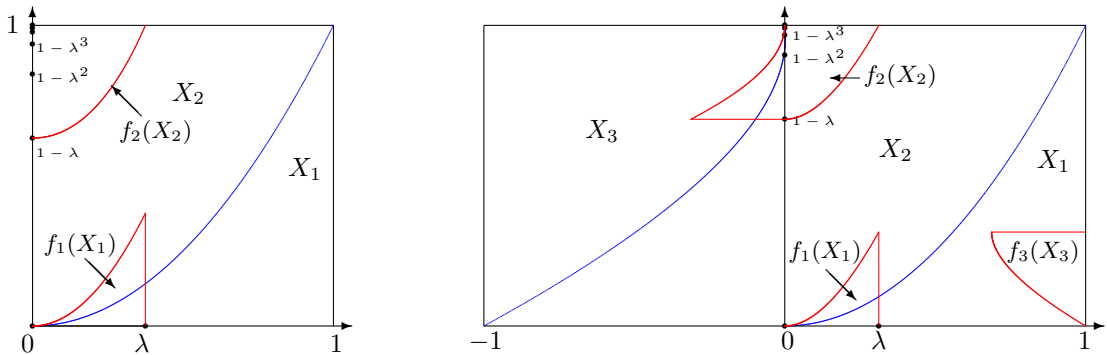


\noindent Now, in Example \ref{ejemploNuevoFebrero2011} all the points of the attractor are non-wandering but not all are limit points.  

\begin{example} \label{ejemploNuevoFebrero2011} \em Let $X := [-1,1] \times [0,1]$ and consider the three open pieces  $X_1 := \{(x,y) \in X : y < x^2, x\geq 0 \}$, $X_2 : = \{(x,y) \in X : y > x^2, x\geq 0\}\setminus\{(0,1)\}\cup\{(x,y) \in X :  y < (x+1)^2,x<0  \}$ and $X_3 : = \{(x,y) \in X : (x+1)^2< y,x<0\}$ (see right frame of Figure \ref{FigureEjemploNuevoFebrero2011}). The set of discontinuities is thus
\[
\Delta = \{(x,y) \in X: y= x^2, x\geq 0\} \cup \{(x,y) \in X : y= (x+1)^2, x\leq 0\}.
\]
Let $\lambda\in(0,1)$ and define on $X$ the maps
$f_1(x,y) = \lambda (x,y)$,  $ f_2(x,y) = \lambda (x,y-1)+ (0,1) $ and $ f_3(x,y) = -\mu (x+1,-y) + (1,0)$, where
$\mu\in(0,3/2-\sqrt{5}/2)$. Then for any piecewise contracting map $f:X\to X$ such that $f|_{X_i}=f_i|_{X_i}$ for all $i\in\{1, 2,3\}$, we have
\begin{center}
$L = \{(0,0),(0,1)\} $ and $ \Omega = \Lambda =  \{(0,1)\} \bigcup { \{(0,1-\lambda^n) : n\in\N, \ n \geq 0\} }$.
\end{center}

\noindent Thus, the attractor $\Lambda$ is countably infinite, all its points are non-wandering and the limit set is finite. Also, the points of $\Omega$ are not in $\text{int}(\widetilde X)$ and therefore the hypothesis of the last assertion of Proposition \ref{OMEGALAMBDA} does not hold. Since the points of $\Omega$ are not periodic, the conclusion also fails. As a final comment, note that if we take $X_3=[-1,0)\times[0,1]$ and $X_2$ as in Example \ref{OpenExample2right}, then the intersection of $\widetilde X$ with any sufficiently small ball centered in one of the point $(0,1-\lambda^n)$, with $n\neq 0$, has an image which intersects itself only once. These points became therefore wandering. 
\end{example}

A continuous system such that $\Lambda\supsetneq\Omega$ or $\Omega\supsetneq  L$ has  some topological expansion, or at least the lack of a  contracting rate, in some part of the space.
For instance $\Omega \supsetneq L$  for bifurcating diffeomorphisms on compact manifolds that exhibit a homoclinic tangency of a dissipative saddle periodic point \cite{palis}.
As said above, piecewise contracting systems may also exhibit $\Lambda \supsetneq \Omega$ in spite of the uniform contracting rate in their pieces.
Hence, the set $\Delta$ of discontinuities acts as  a generator of a peculiar topological expansion,
and piecewise contracting maps have a topological flavor of partial hyperbolicity. In fact, we provide in Example \ref{exampleSegmentoDic2011} a piecewise contracting map whose discontinuities produce a chaotic attractor  contained in $\Delta$.

\section{Dynamics on the attractor}\label{SECDYNATT}

We are now interested in the dynamical properties of the attractors of piecewise contracting maps. 
As mention earlier,  depending on the definition of the map on the set of discontinuities the attractor may fail to be invariant. 
This makes non trivial to define a dynamics in the attractor which is {\it representative} of the asymptotic behaviour of the points of ${\widetilde X}$. One of the goals of this section is to define such a dynamics on the attractor by introducing the concept of {\it ghost orbit}. It will also allow to adapt the notions of recurrence and of transitivity for attractors containing discontinuity points. 

In some case it is possible to define a representative dynamics in the attractor by  multi-defining the
map in the intersection of the attractor with the discontinuities, using its continuous 
extensions\footnote{Recall that the attractor is invariant under these transformations.}. The following example illustrates this fact, and also shows that when the attractor intersects the discontinuities it is not necessarily of infinite cardinality.

\begin{example} \label{ExampleFinito} \em  Consider the rectangle $X = [-1,1]\times [0,1]$ and let $c>0$. Define the set $\Delta:=\{(x,y)\in X: y=x^2-c^2\}\cup\{x=0\}$ and let $X_1 =\{(x,y) \in [0,1]^2 : y < x^2-c^2\}$, 
$X_2 = \{(x,y) \in [-1,0]\times[0,1]: y < x^2-c^2\}$, $X_3 =[-1,0]\times[0,1]\setminus(X_2\cup\Delta)$, and 
$X_4 =[0,1]^2\setminus(X_1\cup\Delta)$, (see left frame of Figure \ref{FIGPERIODIC}).

Let $\lambda\in(0,1)$ and $f_1:\overline{X_1}\to X$ and $f_2:\overline{X_2}\to X$ be defined by $f_1(x,y) := \lambda(x-c,y)$ and $f_2(x,y) =\lambda(x+c,y)$. For all $p\in\N$ let $x_p=c\sum_{k=0}^p\lambda^{-k}$. Take $n\geq 1$ and suppose from now on that $c$ is such that $x_n<1$. Then, we have $f_1(x_{p},0)=(x_{p-1},0)$, $f_2(-x_p,0)=(-x_{p-1},0)$ for all $p\in\{1,\dots,n\}$ and $f_1(x_0,0)=f_2(-x_0,0)=(0,0)$. Now, let $f_3:\overline{X_3}\to X$ and $f_4:\overline{X_4}\to X$ with $f_3(x,y) := \mu(-x,y) + (x_n,0)$ and   $f_4(x,y):= \mu (-x,y) - (x_n,0)$, where $0<\mu<\min\{1-x_n,x_n^2-c^2\}$. Then, we have $f_3(0,0)=(x_n,0)$, $f_4(0,0)=(-x_n,0)$ and the condition on $\mu$ ensures that $f_3(\overline{X_3})\subset X_1$ and  $f_4(\overline{X_4})\subset X_2$.

\begin{figure}[!ht]
\begin{center}
\includegraphics{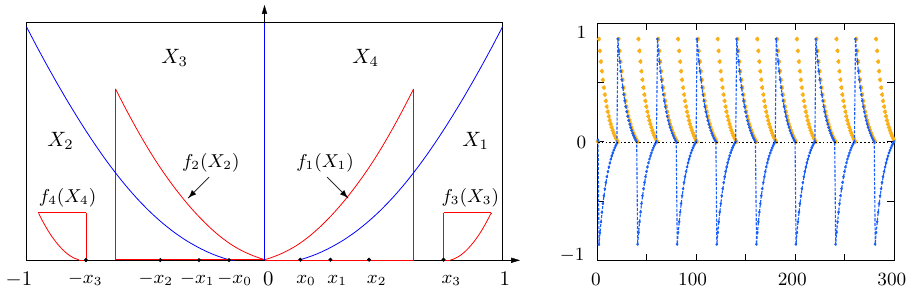}
\caption{
Piecewise contracting map with a finite transitive attractor described by a periodic ghost orbit (Example \ref{ExampleFinito}).
Left frame: the contraction pieces of $f$ and their images with the points of $\Lambda$ for $n = 3$.
Right frame: the $x$ component of $f^k(-c,0.9)$ (blue) and of $f^k(x_n,0)$ (orange) for $k \leq 300$, with $f(0,0) = f_3(0,0)$, $\lambda = 0.9$, $c = 1.5 \times 10^{-2}$ $(n = 18)$ and $\mu = 0.12$.}  
\label{FIGPERIODIC}
\end{center}
\end{figure}

Consider a piecewise contracting map $f: X \to X$ such that $f|_{X_i} = f_i|_{X_i} $ for all $i\in\{1, 2, 3, 4\}$, then one can show that
\begin{equation}\label{LIMITPER}
\omega((x,y))=\{(0,0)\}\cup\{(-x_p,0),(x_p,0)\}_{p=0}^{n}=L=\Lambda\quad\forall\,(x,y)\in{\widetilde X}.
\end{equation}
Alternatively, if we denote $\{\theta_k\}_{k\in\N}$ the periodic sequence such that $\theta_k=1$ for all $k\in\{0,\dots,n\}$,
$\theta_{n+1}=4$, $\theta_k=2$ for all $k\in\{n+2,\dots,2n+2\}$ and $\theta_{2n+3}=3$, then $\Lambda$ can be written as the orbit of $(0,x_n)$ by a multi-valued map:
\begin{equation}\label{ghostperiodic}
\Lambda=\{f^k_{\theta}(x_n,0)\}_{k\in\N}\quad\text{where}\quad f^k_{\theta}(x,y):=f_{\theta_k}\circ\dots\circ f_{\theta_1}\circ f_{\theta_0}(x,y).
\end{equation}
Moreover, there is a neighborhood $U$ of $(x_n,0)$ such that $\lim\limits_{k\to\infty}d(f^k(x,y),f^k_{\theta}(x_n,0))=0$ for any point in $(x,y)\in{\widetilde X}\cap U$. Together with \eqref{LIMITPER}, this implies that all the points of ${\widetilde X}$ follow the orbit of $(x_n,0)$ by $f_\theta$, which covers the whole attractor and is periodic of
period $2n+4$. Note that $f^{n+1}_\theta(x_n,0)=f_4(0,0)$  and $f^{2n+3}_\theta(x_n,0)=f_3(0,0)$. 
Therefore, the asymptotic behaviour of the points of ${\widetilde X}$ is described by a periodic sequence which is not an orbit of the map $f$ (but the orbit of a multi-valued map) and this, independently of the definition of $f$ in $\Delta$.    
See right frame of Figure \ref{FIGPERIODIC}.

\end{example}

Example \ref{ExampleFinito} shows that the asymptotic dynamics of the point of ${\widetilde X}$ is well represented by an  orbit of a multi-valued map constructed with the continuous extensions of the original map. But in the forthcoming examples of this section, a multi-valued map will not be enough to describe the asymptotic behaviours, and this is why we introduce the concept of ghost orbit. A ghost orbit is a generalized orbit: it can be a usual orbit, the orbit of a multi-valued map, or a more sophisticated set whose points are ordered with time sets of order type larger than $\boldsymbol{\omega}$.

\begin{definition}\label{RemarkGhost}{\bf Ghost orbit.} \em Let  $T$ be a well-ordered infinite set with smallest element $t_0$, and consider the map $s:T\to T$ defined by $s(t)=\min\{s\in T: s>t\}$ (by convention we suppose that
$\min\{\emptyset\}=+\infty$). Let $\widehat{T}:=\{t\in T: s^{-1}(t)=\emptyset\}$ and $\hat{s}(t)=\min\{s\in \widehat{T}: s>t\}$. A map $\Phi:T\to X$ is said to be a \em ghost orbit \em of a piecewise contracting map $f:X\to X$, if

\noindent 1) For all $t\in T$ there exists $i\in\{1,\dots,N\}$ such that $\Phi(s(t))=f_i(\Phi(t))$.

\noindent 2) For all $t\in\widehat{T}$ such that $\hat{s}(t)\neq+\infty$ 
there exists $x^\ast\in\overline{\{\Phi(s^k(t))\}_{k\in\N}}\cap\Delta$ such that $\Phi(\hat{s}(t))=f_i(x^\ast)$ for a $i\in\{1,\dots,N\}$. 

\noindent 3) There exists $\epsilon_0>0$ such that for all $\epsilon\in(0,\epsilon_0)$ there is a point $x\in{\widetilde X}$ whose orbit is \em $\epsilon$-close \em to $\Phi$, that is, for an increasing map $k:\widehat{T}\cup\{+\infty\}\to\N\cup\{+\infty\}$  
with $k(t_0)=0$ and $k(+\infty)=+\infty$, we have 

\begin{equation}\label{ECLOSE}
d(f^k(x),\Phi(s^{k-k(t)}(t)))<\epsilon\qquad\forall\,k\in[k(t),k(\hat{s}(t))),\quad\forall\, t\in\widehat{T},
\end{equation}
and
\begin{equation}\label{JUMP}
d(f^{k(\hat{s}(t))}(x),\Phi(s^{k(\hat{s}(t))-k(t)}(t)))\geq\epsilon
\end{equation}
for all $t$ such that $\hat{s}(t)\neq+\infty$.  
\end{definition}

\noindent Note that if $\widehat{T}$ is bounded, then conditions $2)$ and inequality (\ref{JUMP}) do not apply for $t=\max\widehat{T}$. In particular, if $T=\N$, then $\Phi:\N\to X$ is a ghost orbit of $f$ if and only if, it is the orbit of a point of  ${\widetilde X}$, or, it is an orbit of a point of $\Delta$ obtained by successively applying  continuous extensions of $f$ and such that for all $\epsilon>0$ there exists a point of ${\widetilde X}$ which orbit remains at distance smaller than $\epsilon$ of all its points (as the ghost orbit (\ref{ghostperiodic}) of Example \ref{ExampleFinito}). 
The forthcoming Examples \ref{OpenExample2left} and \ref{exampleUncountable1} require ghost orbits whose
time set is $\N\times\N$ endowed with the lexicographic order.

In the following, we are interested in ghost orbits with values in the attractor of piecewise contracting maps. Note first that if $\Phi:T\to X$ is a ghost orbit, then for any $t'\in T$, the map 
$\Phi_{t'}$ defined for all $t\in T$ by $\Phi_{t'}(t)=\Phi(t'+t)$, where $t$ and $t'$ are seen as ordinal numbers and $+$ is the ordinal addition, is a ghost orbit as well. We say that $x\in{\widetilde X}$ belongs to the \em basin of attraction \em of a ghost orbit $\Phi:T\to\Lambda$, if for all $\epsilon>0$ there exists $t'\in T$ and $k\in\N$ such that $f^k(x)$ is $\epsilon$-close to $\Phi_{t'}$.
We say that a ghost orbit $\Phi:T\to\Lambda$ is \em stable\em, if there exists $\epsilon_0>0$ such that for all $\epsilon<\epsilon_0$ there is an open set $U\subset B(\Phi(\tau_0),\epsilon_0)$ such that the orbit of any point of $U\cap{\widetilde X}$ is $\epsilon$-close to $\Phi$. 
A ghost orbit is said to be \em unstable \em if it is not stable and \em repelling \em if the only points with orbit $\epsilon$-close to $\Phi$ belong to $\Phi(T)$.

\begin{example} \label{OpenExample2left} \em Consider the same space $X$ and open pieces $X_1$ and $X_2$ as in Example  \ref{OpenExample2right}. Let $\lambda\in(0,1)$ and define on $\overline{X_1}$ the map $f_1(x,y)=\lambda (x,y) $ and on $\overline{X_2}$ the map
$ f_2(x,y)= \lambda (-x, y) +(1,0) $ (see left frame of Figure \ref{openexample2}). Suppose $f : X \to X$  such that $f|_{X_i}=f_i|_{X_i}$ for all $i\in\{1,2\}$. Then, one can show that
\[
L =  \{(0,0)\} \bigcup {\{(\lambda^n,0) : n \in \N\}} = \Lambda. 
\]
More precisely, for any $(x,y)\in{\widetilde X}$ we have $\omega((x,y))=\{(0,0)\}$ if $y=0$,
and $\omega((x,y))=L$ if $y>0$.

\begin{figure}[!ht]
\begin{center}
\resizebox{13cm}{!}{\includegraphics{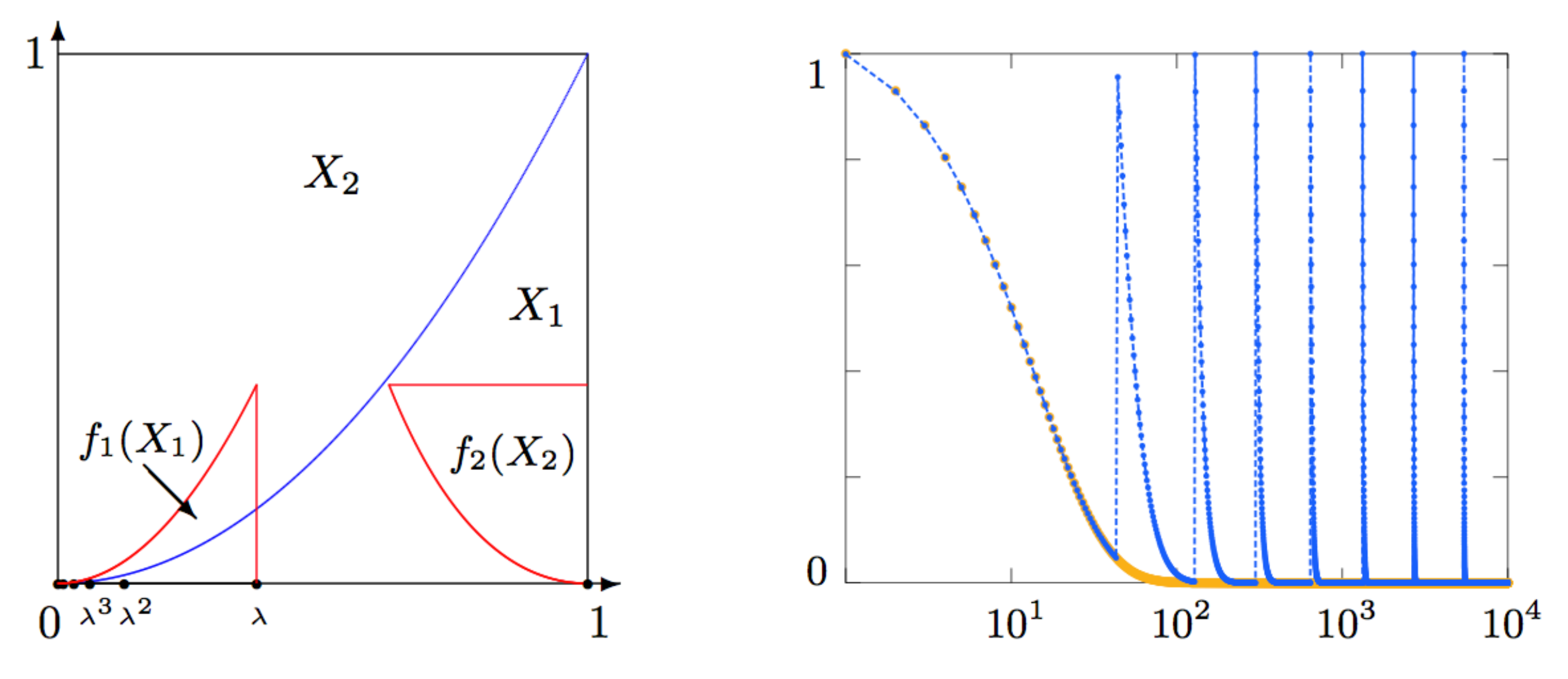}}
\caption{
Piecewise contracting map with an attractor described by a ghost orbit with time set of order type $\boldsymbol{\omega}^2$ (Example \ref{OpenExample2left}).
Left frame: the contraction pieces of $f$ and their images with  points of $\Lambda$.
Right frame: the $x$ component of $f^k(1,0.9)$ (blue) and of $f^k(1,0)$ (orange) for $k \leq 10^4$, with $\lambda = 0.93$ (lin-log plot).}
\label{openexample2}
\end{center}
\end{figure}

\end{example}

\noindent In Example \ref{OpenExample2left}, if $f(0,0)\in\{f_1(0,0),f_2(0,0)\}$ then the attractor is invariant. Also, if we suppose $f(0,0)=f_1(0,0)=(0,0)$ then the attractor can be written as $\Lambda=\{f^n(0,0)\}_{n\in\N}\cup\{f^n(1,0)\}_{n\in\N}$. Any point $(x,y)\in{\widetilde X}$ with $y=0$ has the point $(0,0)$ as $\omega$-limit set, however $(0,0)$ is not stable, since any perturbation of $y$ changes the $\omega$-limit set
of $(x,y)$. Also, although being contained in a contraction piece, $\{f^n(1,0)\}_{n\in\N}$ is a repelling orbit: for any neighborhood $U$ of $(1,0)$, there exists $\epsilon>0$ such that for all $(x,y)\in{\widetilde X}\cap U$ with $y\neq 0$, we have $d(f^{n}(x,y),f^{n}(1,0))>\epsilon$ for some $n\in\N$. 

Actually, the orbits of the points of ${\widetilde X}$ with $y\neq 0$ exhibit a more complicated (stable) asymptotic behaviour than the orbit of the points $(0,0)$ and $(1,0)$. As we can see in the right frame of Figure \ref{openexample2}, and as we are going to show
in the proof of Proposition \ref{GHN2}, after a transient time the orbit of a point $(x,y)\in{\widetilde X}$ with $y\neq 0$ get close to the point $(1,0)$, follows the orbit of this point during a finite time before going back closer to $(1,0)$, and so on.  
There is no way to define the map in $(0,0)$ in order to create an orbit in the attractor with such a recurrent  behaviour. However, we can describe this asymptotic dynamics using a ghost orbit. More precisely, we have the following proposition:

\begin{Proposition}\label{GHN2}
Let $\Phi:\N\times\N\to\Lambda$ be defined by $\Phi(n,k)=f^k(1,0)$ for all $n$ and $k$ in $\N$, where $f$ is a piecewise contracting map of Example \ref{OpenExample2left}.
If $\N\times\N$ is endowed with the lexicographic order, then $\Phi$ is a stable ghost orbit whose basin of attraction are all the points $(x,y)\in\widetilde{X}$ with $y\neq 0$.
If $y=0$ then $(x,y)$ belongs to the basin of attraction of the fixed point $\{f^k_1(0,0)\}_{k\in\N}$, which is an unstable (but not repelling) ghost orbit. 
\end{Proposition}

\begin{proof} If $\N\times\N$ is endowed with the lexicographic order, then $\widehat{T}=\{\tau_n\}_{n\in\N}$ where $\tau_n=(n,0)$ for all $n\in\N$. Also $s^k(\tau_n)=(n,k)$ and $\hat{s}(\tau_n)=\tau_{n+1}$ for all $n$ and $k\in\N$. It is easy to check that $\Phi$ verifies $1)$ and $2)$ of Definition \ref{RemarkGhost}. Now, we show that there exists $\epsilon_0>0$ such that for all $\epsilon<\epsilon_0$ there is an open set $U\subset B(\Phi(\tau_0),\epsilon_0)$ such that  the orbit of any point of $U\cap{\widetilde X}$ is $\epsilon$-close to $\Phi$. This will end to prove that $\Phi$ is a ghost orbit and will show in the same time that it is stable.
For sake of simplicity, and without loss of generality, we make the proof supposing $\lambda<(3-\sqrt{5})/2$ in order to have $f(X_2)\subset X_1$ (as in Figure \ref{openexample2}).

Let $V:=X_1\setminus\{(x,y):y=0\}$ and take $(x,y)\in V\cap{\widetilde X}$. Since $y\neq 0$, the orbit of $(x,y)$ visits $X_2$ infinitely many times, and since $f(X_2)\subset X_1$, it does not stay more that one time step in $X_2$. In other words, the sequence $\{l_n\}_{n\in\N}$ defined by 
\[
l_0=0\quad\text{and}\quad l_{n+1}=\min\{l>l_n:f^l(x,y)\in X_2\}\quad\forall n\in\N
\]  
exists and is such that $f^{l_n+1}(x,y)\in X_1$ for all $n\geq 1$. Now, consider the sequence $\{k_n\}_{n\in\N}$ defined by 
$k_0=0$ and $k_n=l_n+1$ for all $n\geq 1$. This is an increasing sequence which gives the first return times of the orbit of $(x,y)$ in $X_1$. Also the map $k:\widehat{T}\cup\{+\infty\}\to\N\cup\{+\infty\}$ of  Definition \ref{RemarkGhost} will be defined as $k(\tau_n)=k_n$ for all $n\in\N$.

For any $n\in\N$, let $d_n(x,y):=d(f^{k_n}(x,y),(1,0))$, where $d$ denotes the distance induced by the infinite norm of $\R^2$. Then, for all $n\in\N$, we have
\begin{equation*}\label{PROCHE1}
d(f^k(x,y),f^{k-k_n}(1,0))\leq \lambda^{k-k_n} d_n(x,y)\qquad\forall k\in[k_n,k_{n+1}).
\end{equation*}
Now, using the fact that $f^{k_1-1}(x,y)\in X_2$, one can prove that $d_1(x,y)<d_0(x,y)$
if $(x,y)\in B((1,0),1-\lambda)$. This implies that $d_1(f^{k_{n}}(x,y))<d_0(f^{k_{n}}(x,y))$ for all $n\in\N$, or equivalently  
$d_{n+1}(x,y)<d_n(x,y)$ for all $n\in\N$. It follows that 
\begin{equation}\label{PROCHE}
d(f^k(x,y),f^{k-k_n}(1,0))\leq \lambda^{k-k_n} d_0(x,y)\qquad\forall k\in[k_n,k_{n+1})\quad\forall n\in\N, 
\end{equation}
if $(x,y)\in B((1,0),1-\lambda)$.

On the other hand, for all $n\in\N$, if $d_{n+1}(x,y)<(1-\lambda)/2$, then  $d(f^{k_{n+1}}(x,y),f^{k_{n+1}-k_n}(1,0))>(1-\lambda)/2$. It follows that 
\begin{equation}\label{PASPROCHE}
d(f^{k_{n+1}}(x,y),f^{k_{n+1}-k_n}(1,0))>(1-\lambda)/2\quad\forall\, n\in\N,
\end{equation}
if $(x,y)\in B((1,0), (1-\lambda)/2)$, since $d_{n+1}(x,y)<d_{0}(x,y)$ for all $n\in\N$, if $(x,y)\in B((1,0), 1-\lambda)$.

Let $\epsilon_0=(1-\lambda)/2$, $\epsilon<\epsilon_0$ and $U=V\cap B((1,0),\epsilon)$. Then by inequalities (\ref{PROCHE}) and (\ref{PASPROCHE}) any point $(x,y)\in U\cap{\widetilde X}$ is $\epsilon$-close to $\Phi$, and therefore $\Phi$ is a stable ghost orbit. On the other hand, any point $(x,y)\in{\widetilde X}$ with $y>0$ eventually enter in $U$, since $(1,0)\in\omega((x,y))$. These points belong thus to the basin of attraction of $\Phi$. Finally, it is easy to verify that the fixed point $\{f^k_1(0,0)\}_{k\in\N}$ is an unstable ghost orbit, which basin of attraction are all the point $(x,0)$ with $x\neq 0$ and thus is not repelling. 
\end{proof}

For continuous maps, a point is recurrent if it belongs to its $\omega$-limit set.
Following this definition, the attractor of Example \ref{OpenExample2left} does not have recurrent points (except for the point $(0,0)$ in the special case where $f(0,0)=(0,0)$).
In the case of Example \ref{ExampleFinito}, although the dynamics in the attractor being described by a periodic sequence, following the same definition, the attractor contains points that are not recurrent
(these points depend on the definition of $f$ in (0,0)). Once introduced the concept of ghost orbit we can propose a generalized definition of recurrence and of transitivity. 

\begin{definition}\label{definitionRecurrent}{\bf Recurrent point.} \em A point $x\in\Lambda$ is {\it recurrent} if there exists a ghost orbit $\Phi:T\to X$ and $t\in T$ such that $x=\Phi(t)$ and $x\in\overline{\{\Phi(s):s>t'\}}$ for all $t'>t$.\em  
\end{definition}

\noindent Note that if $x=\Phi(t)$ is recurrent then any $y=\Phi(t')$ for a $t>t'$ is also recurrent. Under this definition, all the points of the attractor of Examples \ref{ExampleFinito} and \ref{OpenExample2left} are recurrent.

\begin{definition}\label{definitionTransitive}
{\bf Transitivity.} \em  A compact set $Y\subset\Lambda$ is {\it transitive} if there exists a ghost orbit $\Phi:T\to\Lambda$ such that $Y=\overline{\Phi(T)}$. \em
\end{definition}

\noindent Under this definition, the attractors of Example \ref{ExampleFinito} and \ref{OpenExample2left} are transitive.
The respective attractors of the Example \ref{OpenExample2right} and Example \ref{ejemploNuevoFebrero2011} are also transitive.
The dynamics on the respective attractors is described by the same ghost orbits: the orbit of the unstable fixed point $(0,0)$, the orbit of the stable fixed point $(0,1)$ and the stable ghost orbit $\Phi:\{0,1\}\times\N$ defined by $\Phi(0,k)=(0,0)$ and $\Phi(1,k)=(0,1-\lambda^{k+1})$ for all $k\in\N$. However, the fixed points are the only recurrent points.

We have presented countable attractors so far, but it is known that piecewise contractive maps can exhibit Cantor limit sets. Also, it has been reported that the limit set of some piecewise contracting maps of the plane can be the disjoint union $L = S \cup K $ of a cantor set $K$, supporting a minimal dynamics, with a finite set $S$ of periodic points  \cite{CFLM06}. Such limit sets are decomposable, since they have at least two transitive components. In the following Example \ref{exampleUncountable1}, we show that there exist transitive limit sets which are a disjoint union of the form $S \cup K$ where $S$ is countable. 
This example illustrates also a possible effect of the discontinuities on the characteristics of the recurrent points. Indeed, all the points of $S$ are isolated (non-periodic) recurrent points, whereas for continuous maps such points are necessarily periodic. On the other hand, for a continuous map defined on a compact set $X$, if a point $x\in X$ satisfies $\omega(x)=S\cup K$, where $K$ is a minimal Cantor set and $S$ is a scattered subset of $\omega(x)$, then $S$ is dense in $\omega(x)$ \cite{GRS09}. Example \ref{exampleUncountable1} shows also that this result does not hold for piecewise contracting maps.

\begin{example} \label{exampleUncountable1}\em 
Let $I = [0,1]$, $b \in (0,1)$, $I_1=(b,1)$ and $I_2 =(0,b)$.
Let $\alpha \in (0,1)$ and define on $\overline{I_1}$ the map $g_1 = \alpha (x-b) $ and on $\overline{I_2}$ the map $g_2 = \alpha (x-b) + 1$.
Consider a map $g : I \to I$ such that $g|_{I_i}=g_i|_{I_i}$ for any $i\in\{1,2\}$.
Then, there exist uncountably many $(\alpha,b)$  for which the limit set of $g$ is a minimal Cantor set $K$, and this Cantor set satisfies $\min K=0$ (see \cite[Theorem 6.1]{CFLM06} for an explicit formula to compute such $\alpha$ and $b$).

Let $X=[-1,1]\times[0,1]$, $\phi: [0,b] \to [0,1]$ be some continuous and increasing function, and consider the four open pieces: $X_1 \,=\, \{ (x,y)\in X: b < x \}$, $X_2 \,=\, \{ (x,y)\in X: \, x \in (0,b), \, y < \phi(x) \}$, $X_4 \,=\, \{ (x,y)\in X: \, x \in [-1,0], \, y > (x+1)^2 \}$ and $X_3 \,=[-1,b]\times[0,1]\setminus\overline{(X_4\cup X_2)}$
(see left frame of Figure \ref{Uncountableexample}).

Let $f_1:[0,1]\times[0,1]$ be defined by $f_1(x,y)=(g(x),\lambda y)$, and $f_3:\overline{X_3}\to X$ be defined by $f_3(x,y) =\lambda(x+1,y)+(-1,0)$, where $\lambda\in(0,1)$.

\begin{figure}[!ht]
\begin{center}
\resizebox{15cm}{!}{\includegraphics{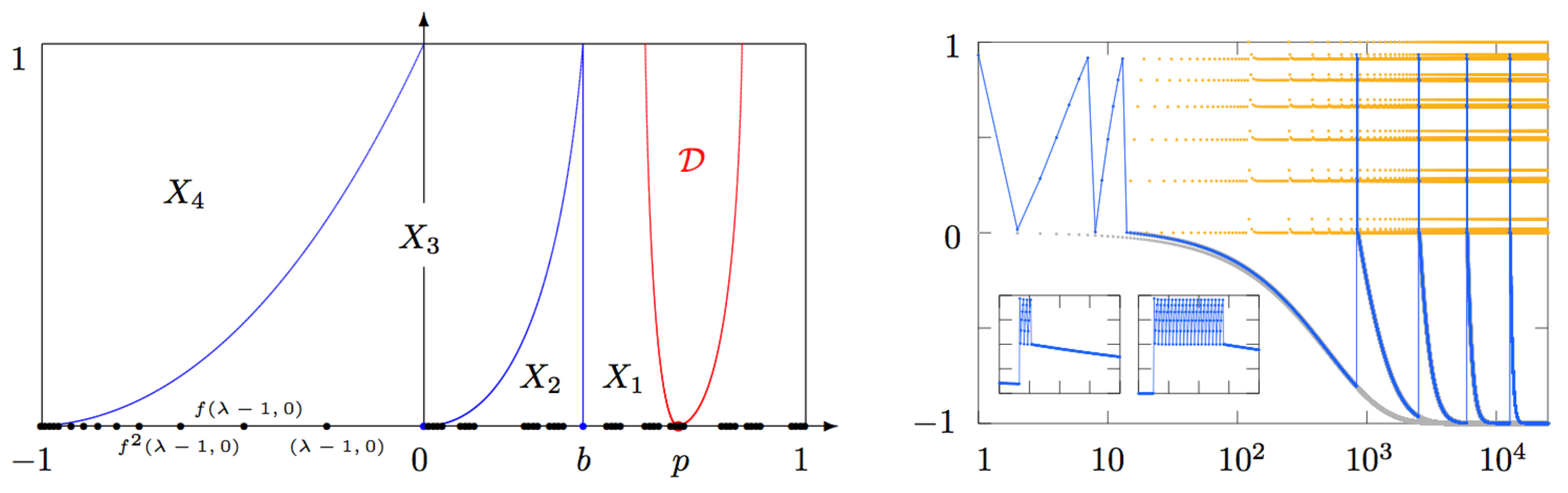}}
\caption{Piecewise contracting map with a transitive limit set $L$, which is the disjoint union of a countably infinite set $S$ with a Cantor set $K$ (Example \ref{exampleUncountable1}).
Left frame: the contraction pieces of $f$ with  points of $L$.
Right frame: the $x$ component of $f^k(0.5,0.9)$ (blue), of $f^k(p,0)$ (orange), and of $f^k(\lambda-1,0)$ (grey) for $k \leq 25 \times 10^4$, with $\lambda = 0.998$, $\alpha = 0.8$, $b \sim 0.911$ and $p \sim 0.934$ (lin-log plot).
The two inner frames, both over a window of 200 time units, are enlargements on the dynamics following the first (leftmost) and fourth (rightmost) return times to $X_1 \cup X_2$ (lin-lin plots).}
\label{Uncountableexample}
\end{center}
\end{figure}

\begin{Proposition}\label{CANTORBEN} There exist a continuous and increasing function $\phi: [0,b] \to [0,1]$, an open set $\mathcal{D}\subset X_1$, and a contracting homeomorphism $f_4:\overline{X_4}\to f_4(\overline{X_4})$, such that for any map $f:X\to X$ satisfying $f|_{X_1} = f_1|_{X_1}$, $f|_{X_2} = f_1|_{X_2}$ and $f|_{X_i} = f_i|_{X_i}$ for any $i \in \{3, 4\}$, the set $\widetilde{X}$ is dense, the limit set $L$ of $f$ satisfies   
\[
L=S\cup K\quad\text{where}\quad S=\overline{\{f^n(\lambda-1,0)\}}_{n\in\N},
\]
and $\omega(x,y)=L$ for any $(x,y)\in\widetilde{X}\cap(X_3\cup X_4\cup\mathcal{D})$. Moreover, $L$ is transitive and all
its points are recurrent.
\end{Proposition}

\noindent The proof of Proposition \ref{CANTORBEN} is given in Section \ref{PROOFCANTOR}. The proof also shows that any point $(x,y)\in\widetilde{X}$ with $y\neq 0$ and whose orbit visits $\mathcal{D}\cup X_3\cup X_4$ is in the basin of attraction
of the stable ghost orbit $\Phi: \mathbb{N} \times \mathbb{N} \to \Lambda$ defined by $\Phi(2n,k) = f^k(\lambda-1,0)$ and $\Phi(2n+1,k) = f^k(p,0)$ for all $n$ and $k\in\N$, where $p$ has an orbit by $f$ which is dense in $K$.
The right frame of Figure \ref{Uncountableexample} shows a numerical simulations of the orbit of $(p,0)$ and 
of the orbit of a point of $\widetilde{X}$ in the basin of attraction of $\Phi$. This orbit accumulates on $S\cup K$ by getting every time closer to the orbit of $(\lambda-1,0)$ and of $(p,0)$. 
\end{example}

\section{Disconnectedness and complexity}\label{SECCOMP}

All the attractors we have encountered until now were totally disconnected.
In this section, we give general conditions ensuring the attractor to have this property, but we also provide examples of piecewise contracting maps with a connected attractor.
We will see that the existence of a connected attractor needs a sufficiently fast growth of the complexity  
to counter-balance its contracting characteristic.
However, as shown by the following theorem, in $\R^n$ and under a reasonable hypothesis on the set of the discontinuity points, the Lebesgue measure of the attractor remains always null.  

\begin{theorem}\label{TheoremMedidaNula} Let $X$ be a compact subset of $\R^n$ and denote by $l_n$ the $n$-dimensional Lebesgue measure. If $f:X\to X$ is a piecewise contracting map such that $l_n(\Delta)=0$, then $l_n(\Lambda)= 0$.
\end{theorem}

\noindent We prove Theorem \ref{TheoremMedidaNula} in Section \ref{proofthmedida}. The condition that the Lebesgue measure of the set $\Delta$ is null may possibly be loosen. But the measurability of this set seems to be important in order to avoid paradoxical decompositions that can lead to the expansion of the Lebesgue measure by a piecewise contracting map \cite{L92}.

Theorem \ref{TheoremMedidaNula} has an immediate consequence on the connectedness of the attractor of a piecewise contracting map defined on a compact subset of $\R$: If the set of the discontinuity points is countable, then the attractor is totally disconnected. Indeed, if $\Delta$ is countable then $l_n(\Delta)= 0$ and it follows from Theorem \ref{TheoremMedidaNula} that $l_n(\Lambda)= 0$. Therefore $\Lambda$ has empty interior, which in $\R$ implies total disconnectedness. In higher dimension or in general metric spaces, the Lebesgue measure (when it exists) does not give information about the connectedness of the attractor. The following theorem gives sufficient conditions for the total disconnectedness of the attractor in compact metric spaces.

\begin{theorem}\label{DISCO} \label{teoremaTotalmenteDesconexo} If a piecewise contracting map  satisfies  at least one of the following hypothesis:

{\em \noindent 1)} there exists $n_0\geq 1$ such that for all $n\geq n_0$, if
$A\cap B\neq\emptyset$ and $A,B\in\mathcal{A}_n$ then $A=B$,

{\em \noindent 2)} the number of atoms of generation $n$, $\# {\mathcal A}_n$, and the contraction rate $\lambda$ satisfy $\lim\limits_{n\to\infty} \# \mathcal{A}_n \, \lambda^n = 0$,

\noindent then its attractor is totally disconnected.
\end{theorem}

\noindent
We prove Theorem \ref{teoremaTotalmenteDesconexo} in Section \ref{proofTheoremTotallyDisconnected}. We have referred earlier to condition 1) as separation property and it is satisfied by any map whose continuous extensions are injective. As a side comment, we mention that similar conditions to 2) allow to obtain an upper bound for the Hausdorff dimension of the attractor (see Proposition \ref{DIMHAUS}). Note that if $N<\lambda^{-1}$, the condition 2) is satisfied, i.e. strongly contracting maps have totally disconnected attractors. For some conformal piecewise contracting maps with polytope pieces, the number of atoms of generation $n$ is a sub-exponential
function of $n$ \cite{KR06b}. This function  is polynomial in some cases where the contraction rate is sufficiently small \cite{LU06}. It implies that these maps always satisfy the condition $2)$ and thus have totally disconnected attractors. As shown by the following example
inspired by \cite{KR06, KR06b}, non-conformal piecewise contracting maps do not always satisfies condition 2) and can exhibit a connected attractor. This example also permits to discuss the optimality of the hypothesis of Theorem \ref{teoremaTotalmenteDesconexo}. 

\begin{example}\label{exampleSegmentoDic2011}\em
Consider in $\mathbb{R}^2$ the compact triangle $T$ with corners in $(0,0)$, $(1,0)$ and $(0,1)$.
We denote by $T_1$ and $T_2$ the two half open triangles, respectively bellow and above the line $x= y$
(here $\Delta_{g}=\{(x,y)\in T : x=y\}$). 
Let $\alpha\in(0,1/2)$. Consider the affinities $g_1$ and $g_2$ defined for all $(x,y)\in T$ by 
$g_1(x,y)=\alpha(x-y,2y)$ and $g_2(x,y)=\alpha(2x,y-x)$. Then, $g_1$ and $g_2$ map respectively $T_1$ and $T_2$ to $\alpha T$ and are such that $g_1(0,0)=g_2(0,0)=(0,0)$, $g_1(1,0)=(\alpha,0)$ and
$g_2(0, 1) = (0, \alpha)$, see Figure \ref{Figure1ExampleCantorDic2011}. Note that any map $g: T \to T$ such that $g|_{T_1} = g_1|_{T_1}$ and $g|_{T_2 } = g_2 |_{T_2 }$ is piecewise contracting, is non conformal and has attractor $\Lambda_{g}=\bigcap_{n\in\N}\alpha^nT=\{(0,0)\}$.

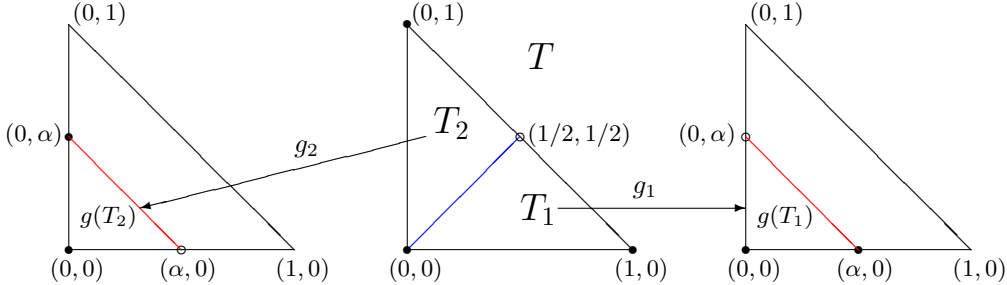
\begin{figure}[h]
\setlength{\unitlength}{0.5truecm}
\begin{center}
\begin{picture}(20,8)(0,0)
\put(-2,1){\line(0,1){6}}
\put(7,1){\line(0,1){6}}
\put(-2,7){\line(1,-1){6}}
\put(-1.7,1.7){\small $g(T_2)$}
\put(-2,1){\circle*{.2}}
\put(-2.5,.3){\small $(0,0)$}
\put(3.5,.3){\small $(1,0)$}
\put(-1.9,7.1){\small $(0,1)$}
\put(-2,4){\circle*{.2}}
\put(-3.7,3.9){\small $(0,\alpha)$}
\put(1,1){\circle{.2}}
\put(.4,.3){\small $(\alpha,0)$}

\put(-2,4){\color[rgb]{1,0,0}\line(1,-1){3}}
\put(7,1){\line(1,0){6}}
\put(16,1){\line(1,0){6}}
\put(7,7){\line(1,-1){6}}
\put(7,1){\color[rgb]{0,0,1}\line(1,1){3}} 
\put(10.2,5.8){\Large $T$}
\put(10,1.8){\Large $T_1$}
\put(7.7,4.1){\Large $T_2$}
\put(7,1){\circle*{.2}}
\put(6.5,.3){\small $(0,0)$}
\put(13,1){\circle*{.2}}
\put(12.5,.3){\small $(1,0)$}
\put(7,7){\circle*{.2}}
\put(7.1,7.1){\small $(0,1)$}
\put(10,4){\circle{.2}}
\put(10.2,3.9){\small $(1/2, 1/2)$}

\put(7.5,4){\Large \vector(-4,-1){7.6}}
\put(4,3.6){$g_2$}
\put(11,2.1){\vector(4,0){5}}
\put(13,2.5){$g_1$}

\put(16,1){\line(0,1){6}}
\put(-2,1){\line(1,0){6}}
\put(16,7){\line(1,-1){6}}
\put(16.3,1.7){\small $g(T_1)$} 
\put(16,1){\circle*{.2}}
\put(15.5,.3){\small $(0,0)$}
\put(21.5,.3){\small $(1,0)$}
\put(16.1,7.1){\small $(0,1)$}
\put(16,4){\circle{.2}}
\put(14.2,3.9){\small $(0,\alpha)$}
\put(19,1){\circle*{.2}}
\put(18.4,.3){\small $(\alpha,0)$}

\put(16,4){\color[rgb]{1,0,0}\line(1,-1){3}}
\end{picture}

\caption{
The non-conformal piecewise affine contracting map $g$ on the triangle $T$ expands the angles at the corner $(0,0)$ and has attractor $\{(0,0)\}$ (Example \ref{exampleSegmentoDic2011}).}
\label{Figure1ExampleCantorDic2011}
\end{center}
\end{figure}

Now we add a dimension by  considering the compact polyhedra $X= T \times[0,1]$. Denote $X_1$  and $X_2$  the two open pieces $T_1 \times [0,1]$ and $T_2 \times [0,1]$, respectively (here $\Delta=\{(x,y,z)\in X : x=y\}$).
Choose two real numbers $\lambda_1, \lambda_2 \in(0,1)$ and let
$\varphi_1$ and $\varphi_2$ be defined for all $z\in[0,1]$ by $\varphi_1(z)=\lambda_1z$ and $\varphi_2(z)=\lambda_2(z-1) + 1$.

Let $f: X \to X$ be a piecewise contracting map such that $f|_{X_i}= f_i|_{X_i}$  for all $i\in\{ 1, 2\}$, where $f_i(x,y,z)=(g_i(x, y), \varphi_i(z))$ for all $(x,y,z)\in X$. In other words, 
\begin{equation}\label{equationF}
f(x, y, z) =\left\{
\begin{array}{ccc}
(g_1(x, y), \varphi_1(z))&\text{ if }& (x,y)\in T_1\\
(g_2(x, y), \varphi_2(z))&\text{ if }& (x,y)\in T_2
\end{array}
\right.
\end{equation}
for all $(x,y,z)\in X_1\cup X_2$ (see Figure \ref{Figure2ExampleCantorDic2011}).

For all $i\in\{1,2\}^\N$ denote $\varphi^n_i:=\varphi_{i_n}\circ\varphi_{i_{n-1}}\circ\dots\circ\varphi_{i_1}$. Then, one can check that for all $i\in\{1,2\}^\N$ and $n\geq 1$, we have 
\[
A_{i_1i_2\dots i_n}=\alpha^nT\times\varphi^n_{i}([0,1]).
\]
It follows that 
\[
\Lambda=\{(0,0)\}\times\Lambda_{\varphi},
\qquad
\text{where}
\qquad
\Lambda_{\varphi}:=\bigcap_{n\in\N}\bigcup_{i\in{\{1,2\}}^\N}\varphi^n_{i}([0,1]).
\]
If we denote by $\varphi:=\{\varphi_1,\varphi_2\}$ the iterated functions system defined by $\varphi_1$ and $\varphi_2$, then $\Lambda_{\varphi}$ is the unique non-empty compact set such that
\[
\Lambda_{\varphi}=\varphi_1(\Lambda_{\varphi})\cup\varphi_2(\Lambda_{\varphi})
\]
and is the attractor of $\varphi$ (see \cite{FAL}, Theorem 9.1). 

The set $\Lambda_{\varphi}$ (and therefore $\Lambda$) is totally disconnected if and only if $\lambda_1 + \lambda_2 < 1$. Actually if  $\lambda_1 + \lambda_2 < 1$ the set $\Lambda_{\varphi}$ is a Cantor set, and if $\lambda_1 + \lambda_2 \geq 1$ it is the interval $[0,1]$. It follows that, in this example, the condition 1) of Theorem \ref{DISCO} is a necessary and sufficient condition. On the other hand, if $\lambda=\max\{\lambda_1,\lambda_2\}>1/2$, the attractor can be totally disconnected and $\lim\limits_{n\to\infty}\#\mathcal{A}_n\lambda^n=\infty$, since $\#\mathcal{A}_n=2^n$, that is, condition 2) is not a necessary condition. But, when the condition 2) does not hold, the attractor can be connected (take $\lambda_1=\lambda_2=1/2$). We note also, that condition 1) does not imply condition 2) in general.

\begin{figure}[!ht]
\begin{center}
\includegraphics{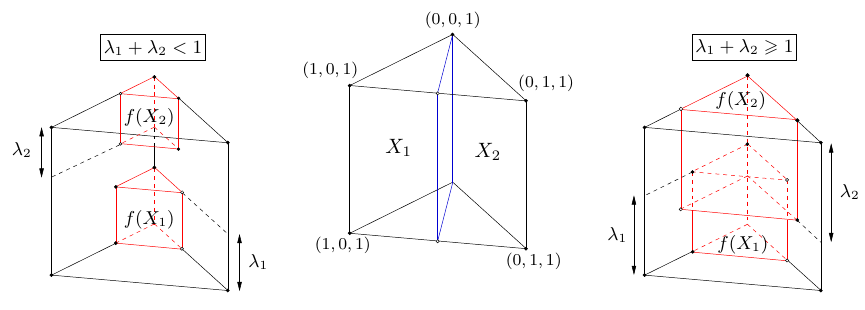}
\caption{
At the centre, the two pieces of contraction (Example \ref{exampleSegmentoDic2011}).
At the left, a piecewise affine contracting map which attractor is a Cantor set.
At the right, a piecewise affine contracting map which attractor is a connected interval.
The projection on the horizontal plane $z = 0$ is the planar piecewise affinity $g$ of Figure \ref{Figure1ExampleCantorDic2011}.}
\label{Figure2ExampleCantorDic2011}
\end{center}
\end{figure}

\em
\end{example}

The following proposition describes in larger details the dynamics of the map of Example \ref{exampleSegmentoDic2011}.

\begin{Proposition}\label{ENTROP} Let $f$ be a piecewise contracting map of Example \ref{exampleSegmentoDic2011}. Then,

\noindent 1) The map $f|_{\widetilde{X}}$ is semi-conjugate to
the shift map $\sigma:\{1,2\}^{\N}\to\{1,2\}^{\N}$ restricted to the set 
\[
\Sigma:=\{i\in{\{1,2\}}^\N : \sigma^k(i)\notin\{1^\infty,2^\infty\}, \forall\, k\geq 1\}\cup\{1^\infty,2^\infty\}. 
\]

\noindent 2) Let $i\in\Sigma$ and $\bar{z}\in[0,1]$ such that $(0,0,\bar{z})\in\Lambda$. Then, the sequence $\{(0,0,\varphi_i^n(\bar{z}))\}_{n\in\N}$ is a ghost orbit of $f$, and any $(x,y,z)\in\widetilde{X}$ such that  $f^n(x,y,z)\in X_{i_{n+1}}$ for all $n\in\N$ belongs to its  basin of attraction.  

\noindent 3) Let $k_0:=\min\{k\in\N:1-\lambda_1^{k}-\lambda_2^{k}>0\}$, then topological entropy of $f|_{\tilde{X}}$ satisfies $h_{top}\geq\frac{\log(2)}{k_0}$. 
\end{Proposition}

\noindent Before proving the proposition we make a comment about point 3). Here, topological entropy refers to the quantity
\[
h_{top}:=\lim_{\epsilon\to 0}\limsup_{n\to\infty}\frac{\log{r(n,\epsilon)}}{n},
\]
where $r(n,\epsilon)$ is the cardinality of largest $(n,\epsilon)$-separated set of $f$ (see \cite{ROB}).
Since all the continuous extensions of $f$ are homeomorphisms, the set $\widetilde{X}$ is dense in $X$ and therefore is not compact. So, although $f$ being continuous on $\widetilde{X}$, we cannot apply directly the classical theorems for continuous maps on a compact set to deduce from 1) that $f|_{\widetilde{X}}$ has  topological entropy greater or equal to that of $(\Sigma,\sigma)$ (that is $\log(2)$). Actually, the map $g|_{\widetilde{T}}$ is also continuous and semi-conjugated to $(\Sigma,\sigma)$. However, all its orbits converge to the point $(0,0)$ and the cardinality of the $(n,\epsilon)$-separated set of $g$ does not increase exponentially with $n$, and therefore $h_{top}=0$ for $g$. This is why the proof of 3) does not rely directly on 1), but consists in using 2) to show that for $f|_{\widetilde{X}}$ the quantity $r(n,\epsilon)$ increases exponentially with $n$.

\begin{proof} 1) We first describe the symbolic dynamics of $g$ on the set ${\widetilde T}:=\bigcap _{n= 0}^{+ \infty} g^{-n}(T\setminus\Delta_{g})$. Let $i:{\widetilde T}\to {\{1,2\}}^\N$ be the map which associates
to each point $(x,y)\in{\widetilde T}$ the sequence $\{i_p(x,y)\}_{p\in\N}$ such that 
$i_p(x,y)=1$ if $g^p(x,y)\in T_1$ and $i_p(x,y)=2$ if $g^p(x,y)\in T_2$ for all $p\in\N$.
Let $Y:=(0,1]\times[0,1]$, and $r:Y\to Y$ such that
$r(u,v)=(\alpha u,s(v))$ for all $(u,v)\in Y$, where $s(v)=2v$ if $v\leq 1/2$
and $s(v)=2v-1$ if $v> 1/2$ (in $(0,1]$ the map $s$ coincides with $2v\mod 1$).
Consider the sets $\Delta_{r}:=\{(u,v)\in Y : u= 0 \text{ or } v=1/2\}$ and ${\widetilde Y}:=\bigcap _{n= 0}^{+ \infty} r^{-n}(Y\setminus\Delta_{r})$. Then, $g|_{{\widetilde T}}$ and $r|_{{\widetilde Y}}$ are topologically conjugated. Indeed, one can check that the map $h:{\widetilde T}\to{\widetilde Y}$ defined by $h(x,y)=(x+y,y/(x+y))$ is a homeomorphism and satisfies $r|_{{\widetilde Y}}\circ h=h\circ g|_{{\widetilde T}}$.

Now let $Y_1:=(0,1]\times[0,1/2)$ and $Y_2:=(0,1]\times(1/2,1]$ and consider the map $j:{\widetilde Y}\to {\{1,2\}}^\N$ which associates to each point $(u,v)\in{\widetilde Y}$ the sequence $\{j_p(u,v)\}_{p\in\N}$ such that $j_p(u,v)=1$ if $r^p(u,v)\in Y_1$ and $j_p(u,v)=2$ if $r^p(u,v)\in Y_2$ for all $p\in\N$. Since the itinerary $j(u,v)$ of any point $(u,v)\in{\widetilde Y}$ is given by the code of $v$ under the map $s$, we have $j(\widetilde Y)=\Sigma$. Using the conjugacy between $g|_{{\widetilde T}}$ and $r|_{{\widetilde Y}}$, one can check that $i(x,y)=j\circ h(x,y)$ for all $(x,y)\in{\widetilde T}$. It follows that $i({\widetilde T})=\Sigma$, since $h({\widetilde T})={\widetilde Y}$. The map $i$ is thus onto (but not one to one) and continuous (by continuity of $r|_{\widetilde Y}$)  and therefore realizes a semi-conjugacy from $g|_{\widetilde T}$ to $\sigma|_{\Sigma}$. Now, since a point $(x,y,z)\in\widetilde{X}$ if and only if $(x,y)\in\widetilde{T}$, from relation \eqref{equationF} we deduce that $f|_{\widetilde{X}}$ is semi-conjugated to $\sigma|_{\Sigma}$ by the map which associates
to each point $(x,y,z)\in\widetilde{X}$ the sequence $i(x,y)$.

\noindent 2) Suppose $\bar{z}$ such that $(0,0,\bar{z})\in\Lambda$ and let $i\in\Sigma$. We first show that the sequence $\{(0,0,\varphi^n_{i}(\bar{z}))\}_{n\in\N}$ is a ghost orbit of $f$. Fist of all $(0,0,\varphi^n_{i}(\bar{z}))=f^n_{i_n}(0,0,\bar{z})$ for all $n\in\N$, and therefore condition 2) of Definition \ref{RemarkGhost} is satisfied and it remains only to show that  \eqref{ECLOSE} holds. Let $\epsilon>0$ and take $n_0\in\N$ such that $\text{diam } A<\epsilon$ for all $A\in\mathcal{A}_n$ and $n\geq n_0$. Let $\theta\in\{1,2\}^\N$ be such that $\bar{z}\in\varphi^{n_0}_{\theta}([0,1])$ and $\sigma^{n_0}(\theta)=i$. Then, $(0,0,\varphi^n_{i}(\bar{z}))\in A_{\theta_1\theta_2\dots\theta_{n_0+n}}$ for all $n\in\N$.  If $i\notin\{1^\infty,2^\infty\}$, the sequence $\theta$ belongs to $\Sigma$ and there exists  $(x',y',z')\in{\widetilde X}$ such that the itinerary of $(x',y')\in{\widetilde T}$ by $g|_{\widetilde T}$ is equal to $\theta$. It follows that the point $(x,y,z):=f^{n_0}(x',y',z')$ belongs to ${\widetilde X}$ and satisfies $f^n(x,y,z)\in A_{\theta_1\theta_2\dots\theta_{n_0+n}}$ for all $n\in\N$. We conclude that  $d(f^n(x,y,z),(0,0,\varphi^n_{i}(\bar{z})))<\epsilon$ for all $n\in\N$. For $i=1^\infty$ (reps. $i=2^\infty$), it is easy to show that the point $(\epsilon,0,\bar{z})$ (reps. $(0,\epsilon,\bar{z})$) belongs to $\widetilde{X}$ and is $\epsilon$-close to $\{(0,0,\varphi^n_{i}(\bar{z}))\}_{n\in\N}$.

Now we describe the basin of attraction of $\{(0,0,\varphi_{i}^{n}(\bar{z})\}_{n\in\N}$. Let $(x,y,z)\in\widetilde{X}$ be such that $f^n(x,y,z)\in X_{i_{n+1}}$ for all $n\in\N$. Then $f^{n}(x,y,z)$ and $(0,0,\varphi_{i}^n(\bar{z}))\in A_{i_1 i_2 \dots i_{n}}$ for all $n\in\N$. It follows that for any $\epsilon>0$, there exists $n_0$ such that $d(f^{n}(x,y,z),(0,0,\varphi_{i}^n(\bar{z})))<\epsilon$ for all $n\geq n_0$. In other words, $f^{n_0}(x,y,z)$ is $\epsilon$-close to the ghost orbit $\{(0,0,\varphi_{i}^{n+n_0}(\bar{z})\}_{n\in\N}$ and therefore belongs to the basin of attraction of $\{(0,0,\varphi_{i}^{n}(\bar{z})\}_{n\in\N}$.

\noindent 3) First note that 
\begin{equation}\label{SEPARPHI}
\varphi^{k}_2(z)-\varphi^{k}_1(z)\geq 1- \lambda_1^{k}-\lambda_2^{k}\qquad\forall\,z\in[0,1],\, 
k\in\N,
\end{equation}
and denote $k_0:=\min\{k\in\N:1-\lambda_1^{k}-\lambda_2^{k})>0\}$. Let $n\geq 1$, $\epsilon_0=(1-\lambda_1^{k_0}-\lambda_2^{k_0})/3$ and $\bar{z}\in[0,1]$ be such that $(0,0,\bar{z})\in\Lambda$. Now let $w\in\{1,2\}^n$ and consider a sequence
$i\in\Sigma$ such that $i_{(p-1)k_0+l}=w_p$ for all $p\in\{1,\dots,n\}$ and $l\in\{1,\dots,k_0\}$.
Since $\{(0,0,\varphi^k_i(\bar{z}))\}$ is a ghost orbit of $f$, there is a point $(x,y,z)\in\widetilde{X}$ such that $d(f^k(x,y,z),(0,0,\varphi^k_i(\bar{z})))<\epsilon_0$ for all $k\in\N$. In particular, if for all $p\in\{1,\dots,n\}$ we denote $\phi^p_w(\bar{z}):=\varphi_{w_p}^{k_0}\circ\varphi^{k_0}_{w_{p-1}}\circ\dots\circ\varphi^{k_0}_{w_1}(\bar{z})$, we have
\begin{equation}\label{XWORD}
d(f^{pk_0}(x,y,z),(0,0,\phi^p_w(\bar{z})))<\epsilon_0 \quad\forall\, p\in\{1,\dots,n\}.
\end{equation}
So we can construct a one-to-one correspondence between $\{1,2\}^n$ and a set $S_{n,\epsilon_0}\subset\widetilde{X}$ such that \eqref{XWORD} holds for any $w\in\{1,2\}^n$ and the corresponding point $(x,y,z)\in S_{n,\epsilon_0}$. On the other hand, if $w\neq w'$, there is (a smallest) $p_0\in\{1,\dots,n\}$ such that $w_p\neq w'_{p}$, and from \eqref{SEPARPHI} it follows that
\begin{equation}\label{SEPAR2}
d((0,0,\phi^{p_0}_{w}(\bar{z})),(0,0,\phi^{p_0}_{w'}(\bar{z})))>3\epsilon_0.
\end{equation}
Let $(x,y,z)$ and $(x',y',z')$ be two different points of $S_{n,\epsilon_0}$, and $w$, $w'$ be their respective corresponding word in $\{1,2\}^n$. Then, using \eqref{XWORD} and \eqref{SEPAR2} and applying the triangular inequality, we obtain that 
\[
d(f^{p_0k_0}(x,y,z),f^{p_0k_0}(x',y',z'))>\epsilon_0,
\]
for a $p_0\in\{1,\dots,n\}$. It follows that for any $n\geq 1$, the set $S_{n,\epsilon_0}$ is a $(k_0n,\epsilon_0)$-separated set for $f$, which cardinality is equal to $2^n$. Using the monotonicity of $r(n,\epsilon)$ in $\epsilon$, we conclude that

\[
h_{top}\geq\limsup_{n\to\infty}\frac{\log{r(n,\epsilon_0)}}{n}\geq\lim_{n\to\infty}\frac{\log{r(nk_0,\epsilon_0)}}{nk_0}\geq\frac{\log(2)}{k_0}.
\]
\end{proof}

\begin{figure}[here]
\begin{center}
\resizebox{15cm}{!}{\includegraphics{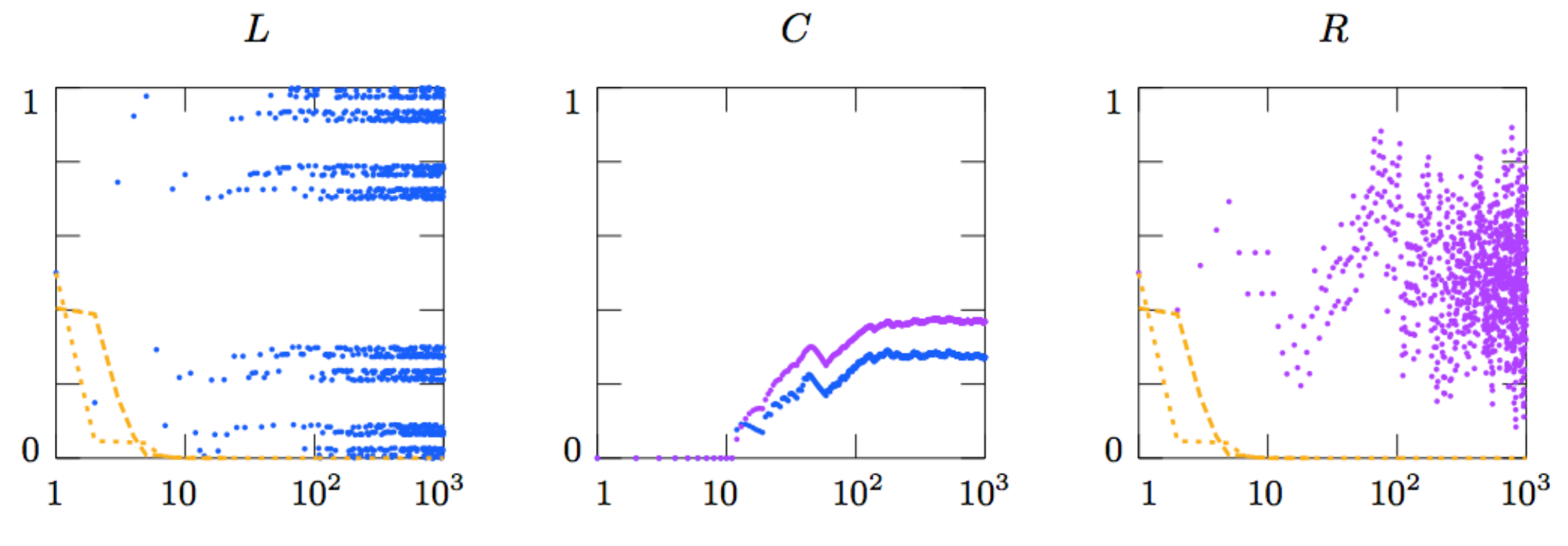}}
\caption{
Piecewise contracting map with positive topological entropy (Example \ref{exampleSegmentoDic2011}).
Left frame $L$: in blue, lin-log plot of the $z$ component of the first 1000 iterates by $f$ of a non periodic point ($x$ and $y$ components are both in orange). Parameters are $\lambda = \lambda_1 = \lambda_2 = 0.3$, and $\alpha = 0.48$ (the attractor is a Cantor set). Right frame $R$: in purple, lin-log plot of the $z$ component of the first 1000 iterates by $f$ of a non periodic point ($x$ and $y$ components are both in orange). Parameters are $\lambda = \lambda_1 = \lambda_2 = 0.8$, and $\alpha = 0.48$ (the attractor is an interval). Centre frame $C$: euclidean distance $\sum_{j=0}^k d(f^j(0.5,y,0.5),f^j(0.5,y+\epsilon,0.5))$ for $k \leq 10^3$, with $y = 0.4$ and $\epsilon \sim 10^{-4}$, for $\lambda = 0.3$ (blue) and $\lambda = 0.8$ (purple).
}
\label{openexample}
\end{center}
\end{figure}

To sum up, the map $f$ of Example \ref{exampleSegmentoDic2011} restricted to ${\widetilde X}$ is semi-conjugated to the shift map on the dense set  $\Sigma$ in the full-shift of two symbols.
Its attractor coincides with its limit set and is a Cantor set or an interval.
In this sense, it is a more complex attractor than whose of the two-dimensional map $g$ or of the example of \cite{KR06} which it is inspired of, which reduces to a single point.
Here (almost) all the orbits of the iterated function system $\varphi$ are ghost orbits attracting the points of ${\widetilde X}$ and the asymptotic dynamics of $f$ has sensitivity to initial conditions.
As an illustration, Figure \ref{openexample} shows numerical simulations of the orbit of a point of $\widetilde{X}$ in the cases where the attractor is a Cantor set ($L$) and where it is an interval ($R$). It also shows the separation of two orbits with close initial conditions ($C$).

\section{Concluding Remarks}

As shown in this paper, piecewise contracting maps can exhibit a large variety of asymptotic behaviours 
supported by attractors ranging from finite to connected and chaotic. 
The property of being contracting in each piece of a dense union of disjoint open pieces determines therefore less characteristics of the asymptotic dynamics than one might expect. It is nevertheless possible to find general conditions ensuring the periodicity, the negligibility in measure, the total disconnectedness, or an upper bound for the Hausdorff dimension of the attractor. Possibly, the currently known sufficient conditions for the genericity of the periodic attractors could also be loosen.

In order to go further in the description of the asymptotic dynamics, one could consider specifics classes of piecewise contractions. As shown by our examples, piecewise affine maps is an interesting class which is already sufficiently general to observe a wide diversity of attractors. It remains then to find the conditions, for instance on the geometry of the discontinuity set and on the way the dynamics accumulates on it, that may impose a given topology and complexity to the asymptotic dynamics.

The study of the dynamics on the asymptotic sets requires the generalization of several basic notions of topological dynamics. 
Our definition of ghost orbit, as the image of a well ordered set, is aimed to give sense to
transitivity and recurrence for attractors containing discontinuity points. Taking into account the observed phenomenology, one could 
study the possible relations between the characteristics of the map (dimension of the space and number of pieces, for instance) and the order type of its ghost orbits.
This and other related questions require further studies on topological 
dynamics with ghost orbits, which could be done in a general framework not necessarily restricted  
to piecewise contracting maps.

We have focused on the topological type and recurrence properties
of the asymptotic sets of piecewise contracting maps, but almost all remains to 
be done concerning the measure theoretical aspects of these systems. Some important results have been obtained in the case of piecewise affine contractions
of the interval, but very few is known about the existence and the characterization of
invariant measures in other cases.

\vspace{5ex}

\noindent {\bf \emph{Acknowledgements.}}
EC was partially financed by CSIC of Universidad de la Rep\'ublica (UdelaR) and ANII of Uruguay.
PG has been supported by grant of Chile FONDECYT 1100764.
AM was partially financed by the scientific cooperation between the IVIC and the UdelaR.
EU gratefully acknowledges SEP-Mexico.
During the development of this work, EC, AM and EU were invited at the Universidad de Vapara\'iso and benefited from the financial support of FONDECYT 1100764.
AM and EU were also invited at the Universidad de la Rep\'ublica and benefited from the financial support of CSIC.

\appendix

\section{Appendix}\label{sectionProofsGeneral}

\subsection{Density of $\widetilde{X}$}

\begin{Proposition} \label{TheoremDensidadXtildePierre}
Let $f:X\to X$ be a piecewise contracting map such that for all $i\in\{1,2,\dots, N\}$ the restricted map $f|_{X_i} : X_i \to f(X_i)$ is an homeomorphism.
Then $\widetilde X := \bigcap _{n= 0}^{+ \infty} f^{-n}(X\setminus\Delta)$ is dense in $X$.\footnote{Proposition \ref{TheoremDensidadXtildePierre} holds also for piecewise continuous maps under the same hypothesis on the restricted maps.}
\end{Proposition}

\begin{proof}
The space $X$ being a compact metric space, it satisfies the Baire property. Therefore, as $\widetilde{X}$ is the countable intersection of the set $\widetilde X_k:= \bigcap _{n= 0}^{k} f^{-n}(X\setminus\Delta)$, it is enough to show that $\widetilde X_k$ is open and dense in $X$ for all $k \in \mathbb{N}$ (which is trivially true for $k=0$). First note that 
\begin{equation}\label{DECOMPXTILDE}
\widetilde X_{k+1}=\bigcup_{i=1}^N (f|_{X_i})^{-1}(\widetilde X_{k}) \quad\forall\,k\in\N.
\end{equation}
Now suppose that $\widetilde X_{k_0}$ is open and dense in $X$ for a $k_0\in\N$ and let $i\in\{1,\dots, N\}$. 
If $f|_{X_i}$ is a homeomorphism, then $\widetilde X_{k_0}\cap f(X_i)$ and $(f|_{X_i})^{-1}(\widetilde X_{k_0})=(f|_{X_i})^{-1}(\widetilde X_{k_0}\cap f(X_i))$ are open sets, and from (\ref{DECOMPXTILDE}) it follows that $\widetilde X_{k_0+1}$ is open.
On the other hand, as $f(X_i)$ is open and $\widetilde X_{k_0}$ is dense in $X$, we have that $\widetilde X_{k_0}\cap f(X_i)$ is dense in $\overline{f(X_i)}$.
Thus, by continuity, $(f|_{X_i})^{-1}(\widetilde X_{k_0})$ is dense in $\overline{X_i}$.
It follows from (\ref{DECOMPXTILDE}) that $ \overline{\widetilde X_{k_0+1}}=\bigcup_{i=1}^N \overline{(f|_{X_i})^{-1}(\widetilde X_{k_0})}=\bigcup_{i=1}^N \overline{X_i}=X$, that is, $\widetilde X_{k_0+1}$ is dense in $X$.  
\end{proof}

\subsection{\bf Proof of Theorem \ref{SEPARATION}}\label{proofTheoremSEPARATION}

The first key point in the proof of Theorem \ref{SEPARATION} is Lemma \ref{LemmaATTRACTOR}, which states that each atom of sufficiently large generation is included in a single continuity piece of $f$.
Since the atoms may be disconnected sets, its proof cannot simply argue that they are necessarily contained in a single continuity piece, just because they do not intersect $\Delta$.

\begin{lemma}\label{LemmaATTRACTOR}
If $\Delta = \emptyset$ or if $d(\Lambda,\Delta)\neq 0$, then there exists $n_0\geq 1$ such that for any atom $A$ of generation $n\geq n_0$ there is a $i\in\{1,\dots,N\}$ such that $A\subset X_i$.
\end{lemma}

\begin{proof} We first assert that there exists a uniform $\delta>0$ such that for any point $x_0 \in \Lambda$, the ball of radius $\delta$ and center $x_0$ is contained in $X_i$, where $X_i$ denotes the unique (open) piece such that $x_0 \in X_i$.
In fact, arguing by contradiction, let us assume that for all $n \in\N$ there exists $i_n\in\{1,\dots,N\}$ and two points $x_n \in \Lambda\cap X_{i_n}$ and $y_n \in X\setminus X_{i_n}$ such that $d(x_n, y_n) < 1/(n+1)$.
Now we replace the sequences $\{x_n\}_{n\in\N}$ and $\{y_n\}_{n\in\N}$ (and we keep the notation $x_n$ and $y_n$ for the respective replacements) by convergent subsequences such that $x_n$ belongs to a constant piece $X_i$ for all $n\in\N$.
Since $d(x_n, y_n) < 1/n$, we obtain that $\{x_n\}_{n\in\N}$ and $\{y_n\}_{n\in\N}$ converge to a same point $x \in \overline {X_i} \cap \overline {X \setminus X_i} \subset \Delta$.
As $\{x_n\}_{n\in\N}\subset \Lambda$, by compactness of $\Lambda$ we have $x \in \Lambda \cap \Delta$, contradicting the hypothesis. We have proved the first assertion.

Now recall that $\Lambda = \bigcap_{n= 1}^{\infty} \Lambda_{n}$, where the sets $\Lambda _n$ are non empty, compact and satisfy $\Lambda_{n+1} \subset \Lambda_n$. This implies the existence of $n_1\geq 1$ such that $ \max\limits_{y \in \Lambda_n} d(y,  \Lambda) < \delta/2$ for all $n \geq n_1$. On the other hand, the maximum diameter of the atoms of generation $n$ converges to 0 when $n$ tends to infinity. Therefore, there exists $n_2 \geq 1$ such that $\mbox{diam}(A) < \delta /2$ for all $A \in {\mathcal A}_n$ with $n \geq n_2$. To end the proof of this lemma, take any atom $A $ of generation $n\geq n_0:=\max\{n_1,n_2\}$ and choose a point $y_0 \in A$. Since $A \subset \Lambda_{n_0}$, we have $d(y_0, \Lambda) < \delta/2 $, that is $d(y_0,x_0) < \delta/2$ for some $x_0 \in \Lambda$. 
Then, for all $y \in A$, we have $d(y, x_0) \leq d(y, y_0) + d(y_0, x_0) \leq \mbox{diam}(A) + \delta/2 < \delta$.
In other words, $A$  is contained in the ball of radius $\delta$ with center at $x_0 \in \Lambda$.
From the first assertion proved at the beginning, the atom $A$ is contained in a unique continuity piece, which ends the proof of the lemma. 
\end{proof}

\begin{proof} [Proof of Theorem \em\ref{SEPARATION}]

{\bf 1)} Let $n\geq n_0$, where $n_0$ is such that Lemma \ref{LemmaATTRACTOR} is satisfied.
Let $A\in\mathcal{A}_{n}$, then there exists $i\in\{1,\dots,N\}$ such that $A\subset X_i$.
Together with the compactness of $A$ and the continuity of $f$ in $X_i$, it implies that $f(A)=\overline{f(A\cap X_i)}$.
Therefore, for all $n\geq n_0$ and $A\in\mathcal{A}_{n}$, the set $f(A)$ is an atom of generation $n+1$ and reciprocally, according to the definition of atom, any atom of generation $n+1$ is the image of an atom of generation $n$. It follows that

\begin{equation}\label{NOSE}
f(\Lambda_n)=\bigcup_{A\in\mathcal{A}_{n}}f(A)
            =\bigcup_{A\in\mathcal{A}_{n+1}}A=\Lambda_{n+1}
\qquad\forall\,n\geq n_0.
\end{equation}
Moreover, since any atom of generation $n+1$ is contained in an atom of generation $n$, for any $A\in{\mathcal A}_{n+1}$ with
$n\geq n_0$, there exists $A'\in{\mathcal A}_{n}$ such that $f(A)\subset A'$.

{\bf 2)} From 1) we deduce that $f:\Lambda_{n_0}\to \Lambda_{n_0}$ induces a transformation $G:{\mathcal A}_{n_0}\to {\mathcal A}_{n_0}$ given by $G(A)=A'$ if $f(A)\subset A'$.
All the orbits of this transformation are eventually periodic, because the collection ${\mathcal A}_{n_0}$ where it acts, is finite.
Thus, the atoms of generation $n_0$ are classified in those that are periodic by the transformation $G$, and those that are not periodic but, anyway, are eventually periodic.
We denote by ${\mathcal P}_{n_0}$ to the finite and not empty sub-collection of atoms in ${\mathcal A}_{n_0}$ that are periodic under the transformation $G$. Let $p_0 \geq 1$ be a multiple of the periods of all the atoms of ${\mathcal P}_{n_0}$, and let $m_0\geq 1$ be the number of different atoms in the family ${\mathcal A}_{n_0}$.
Then we have
$G^{m_0} (A) \in {\mathcal P}_{n_0} $ and $G^{p_0}(G^{m_0}(A)) = G^{m_0}(A)$ for all $A \in {\mathcal A}_{n_0}$.
From the definition of $G$, we deduce that $f^{p_0}(f^{m_0}(A))\subset f^{m_0}(A)$ for all $A\in {\mathcal A}_{n_0}$, or equivalently, 
\begin{equation}\label{PARAFIXPOINT}
f^{p_0}(A)\subset A\quad\forall\, A\in{\mathcal A}_{n_0+m_0}.
\end{equation}
From Lemma \ref{LemmaATTRACTOR}, it follows that $f^{p_0}$ is contracting in any atom of generation larger that $n_0$.
By compactness of the atoms, we deduce from (\ref{PARAFIXPOINT}) that each atom of $\mathcal{A}_{n_0 + m_0}$ contains exactly one periodic point of $f$. We denote $P$ this finite set of periodic points.

{\bf 3)} Recalling that $\Lambda_{n+1}\subset\Lambda_{n}$ for any $n\geq 1$, we can write
$
\Lambda=\bigcap_{n\in\N}\Lambda_n=\bigcap_{j\in\N}\Lambda_{jp_0+n_0+m_0}.
$
Using (\ref{NOSE}) we obtain
\begin{equation}\label{PORFIN}
\Lambda=\bigcap_{j\in\N}f^{jp_0}(\Lambda_{n_0+m_0})
       =\bigcap_{j\in\N}\bigcup_{A\in\mathcal{A}_{n_0+m_0}}f^{jp_0}(A).
\end{equation}
Each point of $P$ is obviously a point of $\Lambda$, and to end the proof of Theorem \ref{SEPARATION} it remains to show  that $\Lambda\subset P$. Let $y\in\Lambda$. By equality (\ref{PORFIN}), for any $j\in\N$ there exists an atom $A_j\in\mathcal{A}_{n_0+m_0}$ such that $y\in f^{jp_0}(A_j)$. Therefore, since $\mathcal{A}_{n_0+m_0}$ is finite, there exists $A\in\mathcal{A}_{n_0+m_0}$ and an increasing sequence $\{j_k\}_{k\in\N}$, such that $y\in f^{j_kp_0}(A)$ for all $k\in\N$. Denote by $y^{\ast} \in P \ $ the $p_0$-periodic point of $f$ in $A$. Since $y$ and $y^{\ast}$ belong both to $f^{j_kp_0}(A)$ for all $k\in\N$, we have $0 \leq d(y^{\ast},y)\leq\text{diam}(f^{j_kp_0}(A))\leq\lambda^{j_kp_0}\text{diam}(A) $ for all $k\in\N$. Thus $y=y^{\ast}$, as wanted.
\end{proof}

\subsection{Proof of Proposition \ref{OMEGALAMBDA}}\label{proofTheoremOMEGALAMBDA}

\noindent To prove that $L\subset\Omega$, take $y\in L$ such that $y\in\omega(x)$ for some $x \in \widetilde{X}$. Then, by definition, there exists a  sequence $\{n_k\}_{k \in \N}$ going to infinity such that $\lim\limits_{k \to \infty} f^{n_k}(x) = y$.
Thus, for all $\epsilon >0$ there exists $k_0\in \N$ such that $f^{n_{k_0+p}}(x)$ belongs to the ball $B(y,\epsilon)$ of center $y$ and radius $\epsilon$ for all $p\in\N$.
In other words, $f^{m_p}(f^{n_{k_0}}(x))\in B(y,\epsilon)$ for all $p\in\N$, where $\{m_p\}_{p \in \mathbb{N}}$ is  defined by $m_p=n_{k_0+p}-n_{k_0}$. Since $f^{n_{k_0}}(x)\in B(y,\epsilon)\cap\widetilde X$, we have $f^{m_p}(B(y, \epsilon) \cap \widetilde X)\cap B(y,\epsilon)\neq\emptyset $ for all $p\in\N$.
According to Definition \ref{definitionNoErrante}, the point $y\in\Omega$. As $\Omega$ is closed, we can conclude that $L\subset \Omega$.

Now we prove that $\Omega\subset\Lambda$. Let $y$ be non wandering.
Then, for all $\epsilon >0$ there exists  $\{n_k(\epsilon)\}_{k \in \N}$ going to infinity such that $f^{n_k(\epsilon)}(B(y,\epsilon)\cap\widetilde X)\cap B(y,\epsilon) $ for all $k \in \N$.
Thus, we can construct a sequence of points $\{y_j\}_{j\in\N}$ and a sequence of natural numbers $\{n_j\}_{j\in\N}$ going to infinity, such that $y_j \in f^{n_j}(B(y, 1/j) \cap \widetilde X) \cap B(y,1/j)$ for all $j\in\N$.
On  one hand, for any $j\in\N$ we have $y_j \in \Lambda_{n_j}$ because $y_j$ belongs to $f^{n_j}(\widetilde X)$ which is contained in $\Lambda_{n_j}$.
On the other hand, since $y_j\in B(y,1/j)$ for all $j\in\N$, we have $\lim\limits_{j \to \infty}y_j=y$.
Therefore, $\lim\limits_{j \to \infty} d(\Lambda_{n_j},y)=0$.
As $\Lambda_{n+1}\subset \Lambda_{n}$ for all $n \in \N$, it follows that $\lim\limits_{n \to \infty} d(\Lambda_n,y)=0$.
We deduce that $d(\Lambda,y) = 0$. So $y \in \Lambda$ and $\Omega\subset\Lambda$.

Finally, we prove the second assertion of Proposition \ref{OMEGALAMBDA}, for which $\widetilde X$
is supposed locally connected. Let $y \in \mbox{int}(\widetilde X)$ be non wandering and let $\epsilon_0 >0$ be such that $B(y,\epsilon)\subset \widetilde X$ is connected for all $\epsilon \in (0,\epsilon_0)$.
Therefore, for all $\epsilon \in (0,\epsilon_0)$ and $n\in\N$

\begin{equation} \label{equalityBolasNoCortanDelta}
f^n(B(y,\epsilon)) \subset \bigcup_{i=1}^{N} X_i,
\end{equation}
and, by continuity, $f^n(B(y,\epsilon))$ is connected.
Since $f^n(B(y,\epsilon)) = \bigcup_{i=1}^{N} (X_i\cap f^n(B(y,\epsilon))$ and the sets $X_i\cap f^n(B(y,\epsilon))$ are pairwise disjoint and open in $f^n(B(y,\epsilon)$, by connectedness of $f^n(B(y,\epsilon))$,  there exists $i_n\in\{1,\dots,N\}$ such that $f^n(B(y, \epsilon)) \subset X_{i_n}$.
It follows that, 
\begin{equation}\label{equationNoErranteTransformaBolas}
f^n(B(y,\epsilon))\subset B(f^n(y), \lambda^n\epsilon) \qquad\forall\, n \in \mathbb{N}, \quad\forall\, \epsilon\in(0, \epsilon_0).
\end{equation}
Let $\epsilon\in(0,\epsilon_0)$ and $\epsilon'=\epsilon/2$. Since $y$ is non wandering and $B(y,\epsilon')\subset\widetilde{X}$, there exists $n_0>0$ such that $f^{n_0}(B(y,\epsilon'))\cap B(y,\epsilon') \neq \emptyset$
and $\lambda^{n_0}<1/3$.
After (\ref{equationNoErranteTransformaBolas}) we have $B(f^{n_0}(y),\lambda^{n_0}\epsilon') \cap B(y,\epsilon') \neq \emptyset$, and it follows that $d(f^{n_0}(y),y)\leq(\lambda^{n_0}+1)\epsilon'<2\epsilon/3$. Let $x\in B(f^{n_0}(y),\lambda^{n_0}\epsilon)$. Then $d(x,y)\leq d(x,f^{n_0}(y))+d(f^{n_0}(y),y)<\lambda^{n_0}\epsilon+2\epsilon/3<\epsilon$. So we have proved that $f^{n_0}(B(y,\epsilon))\subset B(y,\epsilon)$ and using the continuity of $f^{n_0}$ in $B(y,\epsilon)$ we obtain
\begin{equation}\label{YPER}
f^{n_0}(\overline{B(y,\epsilon)})\subset \overline{B(y,\epsilon)}.
\end{equation}
It follows that $B(y,\epsilon)$ contains a unique periodic point of period $n_0$ which is contained in
\[
\bigcap_{\epsilon>0}B(y,\epsilon)=\{y\},
\] 
since \eqref{YPER} is true for any $\epsilon\in(0,\epsilon_0)$.

\subsection{Proof of Proposition \ref{CANTORBEN}}\label{PROOFCANTOR}

Let  $\widetilde{X}_g:= \bigcap _{n= 0}^{+ \infty} g^{-n}(b)$, and $p>b$ be in the (uncountable) set $K\cap{\widetilde X_g}$ and such that $g^n(p)\neq 0$ for all $n\in\N$. Consider the strictly increasing sequence $\{l_i\}_{i\in\N}$ of natural numbers defined inductively by:
\[
l_0 := \min\{n \geq 1   : g^n(p) < b\}\quad\text{and}\quad l_{i+1}
    := \min\{n   >  l_i : g^n(p) < g^{l_i}(p) \}\quad \forall\, i\in\N.
\]
Such sequence exists, since the orbit of $p$ is dense in $K$ and $g^n(p)\neq 0$ for all $n\in\N$.

Choose a real number $\rho \in (0,\lambda)$, and let $\phi: [0,b]\to [0,1]$ be a continuous, increasing and piecewise affine function such that:
\begin{equation}\label{EquationFuncionPhi}
\phi(0)=0, \quad \phi(b)=1  \quad\mbox{and}\quad \phi(g^{l_i}(p)) = \rho^{l_i} \quad \forall\, i \in \N.
\end{equation}
Recall that $f_1:[0,1]^2\to [0,1]^2$ is defined by $f_1(x,y)=(g(x,y),\lambda y)$ for all 
$(x,y)\in[0,1]^2$.

\begin{lemma} \label{lemmaFuncionPhi} {\em 1)} For any $y\in (0,1]$ there exists 
$n_0\geq 1$ such that $f_1^{n_0}(p,y) \in {X}_3$.

\noindent {\em 2)} Let $\{y_j\}_{j\in\N}\in (0,1]^{\N}$ and $\{m_j\}_{j\in\N}$ be defined by $m_j=\min \{n \in\N: \ f_1^n(p,y_j) \in \overline{X}_3\}$ for all $j\in\N$. If $\displaystyle{\limsup_{j\to+\infty} \frac{1}{j}\log y_j}\leq \log \lambda$, then $\lim\limits_{j \to+\infty} m_j=+\infty$.
\end{lemma}

\begin{proof} 1) Suppose that there exists $y\in(0,1]$ such that for all $n\geq 1$ we have $f_1^{n}(p,y) \notin{X}_3\cap[0,b]\times[0,1]$. Then $\phi(g^{l_i}(p))\geq \lambda^{l_i} y$ for all $i\in\N$, since $f_1^n(p,y) = (g^n(p), \lambda^n y)$ for any $n\in\N$. It follows that
$(\rho/\lambda)^{l_i}\geq y>0$ for all $i\in\N$, which is impossible since $\lambda>\rho$ and $\{l_i\}_{i\in\N}$ goes to infinity.

\noindent 2) Fix $j\in\N$ and let $i\in\N$ be such that $l_i\leq m_j< l_{i+1}$. Then, by definition of
$\{l_i\}_{i\in\N}$ we have $g^{m_j}(p )\geq g^{l_i}(p)$. Since $f_1^{m_j}(p,y_j)\in \overline{X}_3$, we have $g^{m_j}(p )<b$ and $\lambda^{m_j}y_j\geq\phi(g^{m_j}(p))\geq \rho^{l_i}\geq \rho^{m_j}$. We deduce that
\[
\frac{1}{j}\log y_j \geq\frac{m_j}{j}(\log\rho - \log\lambda)\quad\forall\, j > 0.
\]
As $\limsup\limits_{j\to+\infty}\frac{1}{j}\log y_j\leq \log \lambda$, there exists $j_0 \geq 1$ such that 
$\log\lambda\geq\frac{m_j}{j}(\log\rho - \log\lambda)$ for all $j > j_0$. We conclude that
\[
m_j\geq\frac{\log\lambda}{\log\rho-\log\lambda}j\quad\forall j>j_0,
\]
and $\lim\limits_{j\to+\infty} m_j=+\infty$.
\end{proof}

Let  $y \in (0,1]$, then from $1)$ of Lemma \ref{lemmaFuncionPhi} there exists $n_0=\min\{n\in\N : f_1^{n}(p,y) \in X_3\}$. Since $p\in{\widetilde X}_g$, the map $f_1^{n}$ is continuous in $(p,y)$ for any $n\in\N$, and there exists $\epsilon(y)>0$ such that the open ball $B((p,y),\epsilon(y))\subset X_1$ and $f_1^{n_0}(B((p,y),\epsilon(y)))\subset X_3$. Furthermore, the radius $\epsilon(y)$ can be chosen small enough to have:
\begin{equation} \label{TOUCHEPAS}
 f_1^n(B((p,y),\epsilon(y)))\cap\{b\}\times[0,1]=\emptyset\qquad\forall n\leq n_0.
\end{equation}
Consider the subset $\mathcal{D}$ of $X_1$ defined by
\begin{equation} 
\mathcal{D}:=\mathcal{B}\setminus\mathcal{C}\quad\text{with}\quad
\mathcal{B}:= \bigcup_{y\in(0,1]} B((p,y),\epsilon(y))\quad\text{and}\quad
\mathcal{C}:=\{(x,y)\in\overline{\mathcal{B}}:(x,y)\notin B((p,y),\epsilon(y))\}.
\end{equation}
The set $\mathcal{D}$ is open ($\mathcal{C}$ is a closed set) and  has the following properties:

\begin{lemma}\label{lemmaConstruccionDeD} \emph{1)} For any $(x,y) \in \mathcal{D}$ there  exists $\tilde{n}_0=\min \{n \in\N: \ f_1^n(x,y) \in X_3\}$ and $(i_0,i_1,\dots,i_{\tilde{n}_0-1})\in\{1,2\}^{\tilde{n}_0}$ such that $f_1^{\tilde{n}_0}(x,y) \in X_3$ and $f_1^n(x,y)$ and $f_1^n(p,y)$ belong to $\overline{X}_{i_n}\setminus\{b\}\times(0,1]$ for all $n\in\{0,\dots,\tilde{n}_0-1\}$.

\noindent \emph{2)} Let $\{(x_j,y_j)\}_{j\in\N}\in\mathcal{D}^\N$ and $\{\tilde{n}_j\}_{j\in\N}$ be defined by $\tilde{n}_j=\min \{n \in\N: \ f_1^n(x_j,y_j) \in X_3\}$ for all $j\in\N$. If $\displaystyle{\limsup _{j \rightarrow + \infty} \frac{1}{j}\log y_j \leq  \log \lambda}$, then
\[
\lim_{j \rightarrow +\infty} \tilde{n}_j = +\infty .
\]
\end{lemma}

\begin{proof} \noindent 1) Let $(x,y)\in\mathcal{D}$, then $(x,y)\in B((p,y),\epsilon(y))$. Now, let $n_0=\min\{n\in\N: f_1^{n}(p,y) \in X_3\}$ and $m_0=\min\{n\in\N: f_1^{n}(p,y) \in \overline{X}_3\}$. Note that $n_0\neq m_0$ if $f_1^{m_0}((p,y))\in\overline{X_2}\cap\overline{X_3}$, and then $n_0>m_0$.
If $n_0=m_0$ or $f_1^{m_0}(x,y)\in\overline{X}_2$, then $f_1^{n_0}(x,y)\in X_3$ and, by equation \eqref{TOUCHEPAS}, there exists $(i_0,i_1,\dots,i_{{n_0}-1})\in\{1,2\}^{n_0}$ such that $f_1^{n_0}(x,y)$ and $f_1^{n_0}(p,y)$ belong to 
$\overline{X}_{i_n}\setminus\{b\}\times(0,1]$ for all $n<n_0$. If $f_1^{m_0}(x,y)\in{X}_3$, then by equation \eqref{TOUCHEPAS}, there exists $(i_0,i_1,\dots,i_{m_0-1})\in\{1,2\}^{m_0}$ such that $f_1^{m_0}(x,y)$ and $f_1^{m_0}(p,y)$ belong to $\overline{X}_{i_n}\setminus\{b\}\times(0,1]$ for all $n<m_0$. The part 1) is proved for a $\tilde{n}_0$ which is either equal to $n_0$ or equal to $m_0$ and such that $\tilde{n}_0\geq m_0$. 

\noindent 2) If $\{(x_j, y_j)\}_{j\in\N}$ in $\mathcal{D}^\N$, then, by definition of $\mathcal{D}$, we have 
$(x_j,y_j)\in B((p,y_j),\epsilon(y_j))$ for all $j\in\N$. If moreover $\limsup\limits_{j\to+\infty} (\log y_j)/j  \leq \log \lambda$, using Lemma \ref{lemmaFuncionPhi}, we obtain  that $\lim\limits_{j\to+\infty}m_j=+\infty$. Since by the proof of $1)$ we know that  $\tilde{n}_j\geq m_j$ for any $j\in\N$, the sequence $\{\tilde{n}_j\}_{j\in\N}$ also goes to infinity. 
\end{proof}

Now, let $f_4:\overline{X}_4\to f_4(\overline{X}_4)$ be defined by $f_4(x,y)=(\lambda\epsilon(\lambda  y)(x+1)+p,\lambda y)$ for each $(x,y)\in\overline{X}_4$. Then $f_4(-1,0)=(p,0)$ and $f_4(X_4)\subset\mathcal{D}$. Moreover, $f_4$ is a contraction, realizes a bijection between the compact sets $\overline{X}_4$ and $f_4(\overline{X}_4)$ and, therefore, is a homeomorphism. 

Recall that $f_3:\overline{X_3}\to X$ is defined by $f_3(x,y) =\lambda(x+1,y)+(-1,0)$. Note that for any $(x,y)\in \overline{X}_3$, we have $f_3(x,y)\in \overline{X}_3\cup \overline{X}_4$ and if $y\neq 0$ then $\min\{n\in\N: f_3^n(x,y)\in \overline{X}_4\}<\infty$ and is a decreasing function of $y$.

Let $f:X\to X$ be such that $f|_{X_1}=f_1|_{X_1}$, $f|_{X_2}=f_1|_{X_2}$, $f|_{X_3}=f_3|_{X_3}$ and $f|_{X_4}=f_4|_{X_4}$. Since, all the continuous extensions of $f$ are homeomorphisms, the set $\widetilde{X}$ is dense.

\begin{lemma}\label{OMEGAGHOST} Let $(x,y)\in\mathcal{D}\cap\widetilde{X}$. There exists a strictly increasing sequence $\{n_j\}_{j\in\N}$ such that: 

\noindent 1) $\lim\limits_{k\to+\infty}(n_{3k+1}-n_{3k})=+\infty$, $\lim\limits_{k\to+\infty}(n_{3k+3}-n_{3k+2})=+\infty$ and $n_{3k+2}=n_{3k+1}+1$.

\noindent 2) For any $\epsilon>0$, there exists $k_0\in\N$ such that for all $k\geq k_0$ we have,
\begin{equation}\label{eqx3}
d(f^n(x,y),f_3^{n-n_{3k}}(0,0))<\epsilon\qquad\forall\,n\in[n_{3k},n_{3k+1}),
\end{equation}
\begin{equation}\label{eqx4}
d(f^{n_{3k+1}}(x,y),(-1,0))<\epsilon,
\end{equation}
and
\begin{equation}\label{eqx1}
d(f^n(x,y),f_1^{n-n_{3k+2}}(p,0))<\epsilon\qquad\forall\,n\in[n_{3k+2},n_{3k+3}).
\end{equation}
\end{lemma}

\begin{proof} Let $(x,y)\in\mathcal{D}\cap\widetilde{X}$. 
By 1) of Lemma \ref{lemmaConstruccionDeD}  the orbit of $(x,y)$ by $f$ enters in $X_3$ from $X_1\cup X_2$ at time $\tilde{n}_0$. Then it remains in $X_3\setminus\{(x,y)\in X: y=0\}$ a finite time (any point of this set is mapped by $f_3$ in $X_4$ in a finite time) and stays one time step in $X_4$ before going back to $\mathcal{D}$ (recall that $f(X_4)\subset\mathcal{D}$). It follows that there exists a strictly increasing sequence $\{n_j\}_{j\in\N}$ such that
for any $k\in\N$, we have 
\[
f^n(x,y)\in X_3\quad\forall\,n\in[n_{3k},n_{3k+1}),\quad f^{n_{3k+1}}(x,y)\in X_4,\quad f^n(x,y)\in X_1\cup X_2\quad\forall\,n\in[n_{3k+2},n_{3k+3}).
\]

\noindent 1)  If we denote $(x_n,y_n):=f^n(x,y)$, then $y_n=\lambda^n y$ for all $n\in\N$. Therefore, by properties of $f_3$, we have $\lim\limits_{k\to\infty}(n_{3k+1}-n_{3k})=+\infty$. Moreover, $\lim\limits_{n \to + \infty} \frac{1}{n}\log y_n = \log \lambda$ and, by 2) of Lemma \ref{lemmaConstruccionDeD}, the sequence defined by $\tilde{n}_0=n_0$ and $\tilde{n}_{k+1}=n_{3k+3}-n_{3k+2}$ for all $k\in\N$ goes to infinity. Now, since $f(X_4)\subset\mathcal{D}$, we can chose $\{n_j\}_{j\in\N}$ such that $n_{3k+2}=n_{3k+1}+1$.

\noindent 2)  Since $f^{n_{3k+2}}(x,y)\in \mathcal{D}$, by 1) of Lemma \ref{lemmaConstruccionDeD}, for any $k\in\N$ and for all $n\in[n_{3k+2},n_{3k+3}]$ we have,
\begin{eqnarray}\label{DEM1}
d(f^n(x,y), f_1^{n-n_{3k+2}}(p,0))&\leq& d(f^n(x,y), f_1^{n-n_{3k+2}}(p,y_{n_{3k+2}})) + d(f_1^{n-n_{3k+2}}(p,y_{n_{3k+2}}), f_1^{n-n_{3k+2}}(p,0))\nonumber\\
&\leq&\lambda^{n-n_{3k+2}}d(f^{n_{3k+2}}(x,y),(p,y_{n_{3k+2}})) + \lambda^{n-n_{3k+2}}d((p,y_{n_{3k+2}}),(p,0))\nonumber\\
&\leq& \lambda^{n-n_{3k+2}}d(f^{n_{3k+2}}(x,y),(p,y_{n_{3k+2}}))+\lambda^n y\nonumber\\
&\leq& \lambda^{n-n_{3k+2}}d(f^{n_{3k+2}}(x,y),(p,0))+2\lambda^n y.
\end{eqnarray}
In particular, applying \eqref{DEM1} for $n=n_{3k+3}$, we obtain 
\[
d(f^{n_{3k+3}} (x, y), K )\leq \lambda^{\tilde{n}_{k+1}}d(f^{n_{3k+2}}(x,y),(p,0))+2\lambda^{n_{3k+3}}\quad\forall\, k\in\N.
\]
As $\{\tilde{n}_k\}_{k\in\N}$ goes to infinity, $f^{n_{3k+3}} (x, y)\in X_3$ and $f_1^{\tilde{n}_k}(p,0)\in K$ for all $k\in\N$,
we have 
\[
\lim _{k \to + \infty} f^{n_{3k+3}}(x, y)\in K\cap\overline{X}_3=\{(0,0)\}.
\]
On the other hand, since $f^{n_{3k+1}}(x,y)\in X_4$ for all $k\in\N$ and $\lim\limits_{k\to+\infty}{y_{3k+1}}=0$, we have
\[
\lim_{k\to+\infty}f^{n_{3k+1}}(x,y)=(-1,0).
\]
Let $\epsilon>0$, then there exist $k_0\in\N$ such that   
\[
d(f^{n_{3k}}(x,y),(0,0))<\epsilon,\quad 
d(f^{n_{3k+1}}(x,y),(-1,0))<\epsilon\quad\text{and}\quad
\frac{\lambda^{n_{3k+2}}}{1-\lambda}<\frac{\epsilon}{2}\qquad \forall\, k\geq k_0,
\]
which, in particular, proves \eqref{eqx4}. From now on we suppose $k\geq k_0$. Since $f_4(f^{n_{3k+1}}(x,y))=f^{n_{3k+2}}(x,y)$, using \eqref{DEM1}, we obtain 
\[
d(f^n(x,y), f_1^{n-n_{3k+2}}(p,0))\leq \lambda^{n-n_{3k+2}}d(f_4(f^{n_{3k+1}}(x,y)),f_4(-1,0))+2\lambda^n y\leq\lambda\epsilon+2\lambda^{n_{3k+2}}<\epsilon,
\]
for any $n\in[n_{3k+2},n_{3k+3}]$, and \eqref{eqx1} is proven. Finally, \eqref{eqx3} follows from
\[
d(f^n(x,y),f_3^{n-n_{3k}}(0,0))<\lambda^{n-n_{3k}}d(f^{n_{3k}}(x,y),(0,0))<\epsilon\quad\forall\,n\in[n_{3k},n_{3k+1}].
\]
\end{proof}

From Lemma \ref{OMEGAGHOST} it follows immediately that $\omega(x,y)=S\cup K$ for any 
$(x,y)\in\mathcal{D}\cap\tilde{X}$, and therefore for any $(x,y)\in\widetilde{X}$ with $y\neq 0$ which orbit
visits $\mathcal{D}\cup X_3\cup X_4$. Moreover, all these points are in the basin of attraction
of the stable ghost orbit $\Phi: \mathbb{N} \times \mathbb{N} \to \Lambda$ defined by $\Phi(2n,k) = f_3^k(\lambda-1,0)$ and $\Phi(2n+1,k) = f_1^k(p,0)$ for all $n$ and $k\in\N$. This prove in particular that $S\cup K$ is transitive and all its points are recurrent. Moreover, it is easy to show that $\omega(x,y)=(-1,0)$ for any $(x,y)\in\widetilde{X}\cap X_3$ with $y=0$ and $\omega(x,y)=K$ for any $(x,y)\in\widetilde{X}\cap (X_1\cup X_2)$ which orbit never enters in $\mathcal{D}$. It follows that $L=S\cup K$.

\subsection{Proof of Theorem \ref{TheoremMedidaNula}}\label{proofthmedida}
First notice that, for any set $S\subset X$, we have $\overline{f(S \cap X_i)} = f_i(\overline{S} \cap \overline{X_i})$,
where $f_i:\overline{X}_i\to X$ is the continuous extension of $f|_{X_i}$ to $\overline{X_i}$. Therefore, for all $k\in\N$ we can 
write $\Lambda_{k+1}$ as
\[
\Lambda_{k+1} =  \bigcup_{i= 1}^N\overline{f(\Lambda_k \cap X_i)}=\bigcup_{i= 1}^N{f_i(\Lambda_k \cap \overline{X_i})}  = \Big(\bigcup_{i= 1}^N{f (\Lambda_k \cap {X_i})}\Big) \, \bigcup \, \Big(\bigcup_{i= 1}^N{f_i(\Lambda_k \cap {\partial X_i})}\Big).
\]

Let $\mathcal{H}^s$ be the $s$-dimensional Hausdorff measure and let $k\in\N$, then
\begin{equation}\label{HLK}
 \mathcal{H}^s(\Lambda_{k+1}) \leq \mathcal{H}^s\Big(\bigcup_{i= 1}^N{f (\Lambda_k \cap {X_i})}\Big) \, + \, \mathcal{H}^s\Big(\bigcup_{i= 1}^N{f_i(\Lambda_k \cap {\partial X_i})}\Big).
\end{equation}
For the first term on the right hand side, we have
\[
 \mathcal{H}^s\Big(\bigcup_{i= 1}^N{f (\Lambda_k \cap {X_i})}\Big)  \leq  \sum_{i=1}^{N} \mathcal{H}^s \Big( f (\Lambda_k \cap {X_i}) \Big) \leq  \sum_{i=1}^{N} \lambda^s \mathcal{H}^s( \Lambda_k \cap {X_i} ) =\lambda^s \mathcal{H}^s( \Lambda_k\setminus\Delta ),
\]
where the last inequality follows from the fact that  $f$ is a $\lambda$-Lipchitz function in any set $X_i$, and the last equality  follows from the disjointness of the Borel sets $\Lambda_k \cap {X_i}$. For the second term on the right hand side of \eqref{HLK}, we have:
\[
\mathcal{H}^s\Big(\bigcup_{i= 1}^N{f_i(\Lambda_k \cap {\partial X_i})}\Big)  \leq  \sum_{i=1}^{N} \mathcal{H}^s \big( f_i(\Lambda_k \cap {\partial X_i}) \big) \leq \sum_{i=1}^{N} \lambda^s \mathcal{H}^s(\Lambda_k \cap {\partial X_i})\leq 
\lambda^s \sum_{i=1}^{N}\mathcal{H}^s({\partial X_i}),
\]
since for any $i\in\{1,\dots,N\}$ the map $f_i$ is also a $\lambda$-Lipchitz function.
Therefore, we have proved that
\[
 \mathcal{H}^s(\Lambda_{k+1})  \leq  \lambda^s (\mathcal{H}^s( \Lambda_k ) - \mathcal{H}^s(\Delta) + \sum_{i=1}^{N}\mathcal{H}^s({\partial X_i}))\quad\forall\,k\in\N.
\]
Let $\delta=\sum_{i=1}^{N}\mathcal{H}^s({\partial X_i})- \mathcal{H}^s(\Delta)$, then by induction we obtain:
\[
 \mathcal{H}^s(\Lambda_{k+1})  \leq  \lambda^{sk} \mathcal{H}^s( \Lambda_1 ) + \delta\sum_{j=1}^{k} \lambda^{sj}\leq\lambda^{sk} \mathcal{H}^s( X ) + \frac{\lambda^{s}(1-\lambda^{sk})}{1-\lambda^{s}}\delta\quad\forall\,k\in\N.
\]
Since $\Lambda\subset \Lambda_{k+1}$ for any $k$, it follows that
\[
 \mathcal{H}^s(\Lambda)  \leq  \lim_{k\to\infty}\lambda^{sk} \mathcal{H}^s( X ) + \frac{\lambda^{s}}{1-\lambda^{s}}\delta.
\]
Recalling that $X\subset\R^n$ and that $\mathcal{H}^n(A)=l_n(A)$ for any Borel set $A\subset\R^n$, we deduce that
\[
l_n(\Lambda)  \leq   \frac{\lambda^{n}}{1-\lambda^{n}}\left(\sum_{i=1}^{N}l_n({\partial X_i})- l_n(\Delta)\right),
\]
since $l_n(X)<\infty$. It follows that $l_n(\Lambda)=0$, if $l_n(\Delta)=0$. 

\subsection{Proof of Theorem \ref{teoremaTotalmenteDesconexo}}\label{proofTheoremTotallyDisconnected}

If hypothesis 1) of Theorem \ref{teoremaTotalmenteDesconexo} is verified, then, for any $n\geq n_0$, the different atoms of $\mathcal{A}_{n}$ form a finite partition of  $\Lambda_n$ into compact pairwise disjoint sets. Therefore, for any $n\geq n_0$, a connected component of $\Lambda$ belongs to a single atom of generation $n$, and thus it is of diameter arbitrarily small.

The total disconnectedness of the attractor, in the case where hypothesis 2) of Theorem \ref{teoremaTotalmenteDesconexo} holds, is a corollary of the following proposition:

\begin{Proposition}\label{DIMHAUS} Let $f:X\to X$ be a piecewise contracting map  such that 
\begin{equation}\label{CONDDIMH}
\lim_{n\to\infty}\#\mathcal{A}_n\lambda^{ns}=0
\end{equation}
for some $s>0$. Then the $s$-dimensional Hausdorff measure of the attractor of $f$ satisfies $\mathcal{H}^s(\Lambda)=0$. \end{Proposition}

\begin{proof} Let $\alpha> 0$ and $n_0$ be such that $\#\mathcal{A}_n\lambda^{ns}<\alpha$ for any $n\geq n_0$. Fix $\delta>0$, and let $n_1$ be such that for any $n>n_1$ we have $\mbox{diam}(A)<\delta$ for all $A\in\mathcal{A}_n$. Let $m=\max\{n_0,n_1\}$, then the atoms of $\mathcal{A}_m$ form a $\delta$-cover of $\Lambda$ and we have
\[
\mathcal{H}^{s}_\delta(\Lambda):=\inf\left\{\sum_{i\in I} \mbox{diam}(U_i)^{s} : \{U_i\}_{i\in I}\ \delta\text{-cover of } \Lambda\right\}\leq\sum_{A\in\mathcal{A}_m}\mbox{diam}(A)^{s}\leq\#\mathcal{A}_m\lambda^{ms}\mbox{diam}(X)^{s}<\alpha\mbox{diam}(X)^{s}.
\] 
Here we have used the fact that $A_{i_1i_2\dots i_n}=\overline{f(A_{i_1i_2\dots i_{n-1}}\cap X_{i_n})}$ for any atom $A_{i_1i_2\dots i_n}\in\mathcal{A}_n$, wich implies  
\[
\mbox{diam}(A_{i_1i_2\dots i_n})\leq\lambda\mbox{diam}(A_{i_1i_2\dots i_{n-1}})\leq\lambda^{n-1}\mbox{diam}(A_{i_1})\leq\lambda^n\mbox{diam}(X_{i_1})\leq\lambda^n\mbox{diam}(X)\quad \forall\, n> 1.
\]
Since $\delta$ is arbitrary, we have  $\mathcal{H}^{s}_\delta(\Lambda)<\alpha\mbox{diam}(X)^{s}$ for any $\delta>0$ and it follows that
\begin{equation}\label{HYPHAUS}
\mathcal{H}^{s}(\Lambda):=\lim_{\delta\to 0}\mathcal{H}^{s}_\delta(\Lambda)\leq\alpha\mbox{diam}(X)^{s}.
\end{equation}
As \eqref{HYPHAUS} is true for any $\alpha>0$, we deduce that $\mathcal{H}^{s}(\Lambda)=0$.
\end{proof}
If \eqref{CONDDIMH} holds for $s=1$, then $\mathcal{H}^1(\Lambda)=0$, which implies that $\Lambda$ is totally disconnected (see proof of Proposition 2.5 of \cite{FAL}).
Another immediate corollary of Proposition \ref{DIMHAUS} is that \eqref{CONDDIMH} implies that the Hausdorff  dimension of the attractor of $f$ satisfies $\dim_{H}(\Lambda):=\inf\{s_0\geq 0:\mathcal{H}^{s_0}(\Lambda)=0\}\leq s$.

\bibliographystyle{amsplain}

\end{document}